\documentclass[11pt]{amsart}
\usepackage[a4paper,hmargin=3cm,vmargin=3cm]{geometry}
\usepackage{amsfonts,amssymb,amscd,amstext,amsthm}
\usepackage[pdftex]{graphicx}
\usepackage[dvips]{epsfig}
\usepackage[T1]{fontenc}
\usepackage[english]{babel}
\usepackage[utf8]{inputenc}
\usepackage[bookmarksnumbered,plainpages]{hyperref}
\usepackage{indentfirst}
\usepackage{color}
\usepackage{url}

\usepackage{fancyhdr}
\usepackage{times}

\usepackage[shortlabels]{enumitem}
\usepackage{imakeidx}
\usepackage{titlesec}
\usepackage{mathrsfs}
\usepackage{stmaryrd}
\usepackage{textfit}
\usepackage{lipsum}

\DeclareMathOperator{\sign}{sign}
\def \R{\mbox{${\mathbb R}$}}
\def \S{\mbox{${\mathbb S}$}}

\def \J{\mbox{$\mathrm{J}$}}
\def \h{\mbox{$\mathfrak{h}$}}

\titleformat{\section}
{\filcenter\bfseries\large} {\thesection{.}}{0.2cm}{}
\titleformat{\subsection}[runin]
{\bfseries} {\thesubsection{.}}{0.15cm}{}[.]
\titleformat{\subsubsection}[runin]
{\em}{\thesubsubsection{.}}{0.15cm}{}[.]

\usepackage[up,bf]{caption}

\newtheorem{theorem}{Theorem}[section]

\newtheorem{proposition}[theorem]{Proposition}

\newtheorem{remark}[theorem]{Remark}

\newtheorem{definition}[theorem]{Definition}

\theoremstyle{definition}

\numberwithin{equation}{section}
\numberwithin{figure}{section}

\begin{document}
\fancyhead[LO]{Helicoidal surfaces of prescribed mean curvature}
\fancyhead[RE]{Aires E. M. Barbieri}
\fancyhead[RO,LE]{\thepage}

\thispagestyle{empty}

\begin{center}
{\bf \LARGE Helicoidal surfaces of prescribed mean curvature in $\R^3$}
\vspace*{5mm}

\hspace{0.2cm} {\Large Aires E. M. Barbieri} 
 
\end{center}


\vspace*{8mm}

\begin{quote}
{\small
\noindent {\bf Abstract}\hspace*{0.1cm}
    Given a function $\mathcal{H} \in C^1(\S^2)$, an $\mathcal{H}$-surface $\Sigma$ is a surface in the Euclidean space $\R^3$ whose mean curvature $H_\Sigma$ satisfies $H_\Sigma = \mathcal{H} \circ \eta$, where $\eta$ is the Gauss map of $\Sigma$. The purpose of this paper is to use a phase space analysis to give some classification results for helicoidal $\mathcal{H}$-surfaces, when $\mathcal{H}$ is rotationally symmetric, that is, $\mathcal{H} \circ \eta = \mathfrak{h} \circ \nu$, for some $\h \in C^1([-1,1])$, where $\nu$ is the angle function of the surface. We prove a classification theorem for the case where $\h(t)$ is even and increasing for $t \in [0,1]$. Finally, we provide examples of helicoidal $\mathcal{H}$-surfaces in cases where $\h$ vanishes at some point.
}
\\

{\small
\noindent {\textbf{Mathematics Subject Classification:}}\hspace*{0.1cm}
    53A10, 53C42, 34C05, 34C40.
}
\\
{\small
\noindent {\textbf{Keywords:}}\hspace*{0.1cm}
    helicoidal surfaces, prescribed mean curvature.
}

\end{quote}


\section{Introduction}

In this article, we study the existence and classification of helicoidal surfaces $\Sigma$ of $\R^3$ whose mean curvature $H_\Sigma$ is given as a prescribed function of its Gauss map $\eta: \Sigma \to \S^2$. Specifically, given $\mathcal{H} \in C^1(\S^2)$, we are interested in finding surfaces $\Sigma$ that satisfy
$$H_\Sigma = \mathcal{H} \circ \eta.$$
Any such $\Sigma$ will be called a surface of \textit{prescribed mean curvature} $\mathcal{H}$, or simply, an \textit{$\mathcal{H}$-surface}.

The study of $\mathcal{H}$-hypersurfaces in $\R^{n+1}$ is motivated by the classical contributions of Alexandrov and Pogorelov in the 1950s (see \cite{ALEXANDROV,POGORELOV}). They started to investigate the existence and uniqueness of ovaloids in $\R^{n+1}$ determined by a prescribed curvature function through its Gauss map.

When specific functions $\mathcal{H}$ are chosen as prescribed, some geometric theories receive deeply research. The following examples are emphasized:

\begin{enumerate}
    \item Surfaces of constant mean curvature, or simply, CMC surfaces. In this case, $\mathcal{H} \equiv H_0 \in \R$ is a constant. When $H_0 = 0$, they are minimal surfaces;

    \item $\lambda$-translators, which corresponds to the case $\mathcal{H}(x) = \langle x, v \rangle + \lambda$, where $v \in \S^n$ and $\lambda \in \R$. The vector $v$ is called the density vector. When $\lambda = 0$, we have translating solitons of the mean curvature flow.
\end{enumerate}

The theory of complete, non-compact $\mathcal{H}$-surfaces for general choice of $\mathcal{H}$ has been recently developed by A. Bueno, J. Gálvez and P. Mira \cite{GLOBALGEO,ROTATIONAL}, taking as starting point the theories of CMC surfaces and translating solitons. In the first article, the authors started the development of the global theory of complete $\mathcal{H}$-hypersurfaces in $\R^{n+1}$, demonstrating that, under mild symmetry and regularity assumptions on the function $\mathcal{H}$, the theory of $\mathcal{H}$-hypersurfaces sometimes admits a uniform treatment that resembles the constant mean curvature case.

In \cite{ROTATIONAL}, focusing on the study of rotational $\mathcal{H}$-hypersurfaces in $\R^{n+1}$, in case that the function $\mathcal{H} \in C^1(\S^n)$ is rotationally symmetric, i.e. $\mathcal{H}(x)=\h(\langle x, e_{n+1} \rangle)$ for some $\h \in C^1([-1,1])$, the same authors treated the resulting  ordinary differential equation as a nonlinear autonomous system and conducted a qualitative study of its solutions through a phase space analysis. The aim was  to provide classification results for these hypersurfaces.

Motivated by these ideas, our objective in this article is to use a similar phase space analysis to study, classify and exhibit some helicoidal surfaces of prescribed mean curvature in $\R^3$.

First, in Section \ref{section1}, we define a helicoidal surface $\Sigma$ in $\R^3$ (Def. \ref{defhelicoidal}) and explicitly calculate in \eqref{meancurvaturehelicoidal} the mean curvature of $\Sigma$ in terms of its profile curve. Since the mean curvature of a helicoidal surface does not depend on its rotational parameter, we focus on the case where $\mathcal{H} \in C^1(\S^2)$ is rotationally symmetric, that is, up to Euclidian change of coordinates, there is a function $\h \in C^1([-1,1])$ such that $\mathcal{H}(x) = \h(\langle x, e_3 \rangle)$. Then, the mean curvature of a helicoidal $\mathcal{H}$-surface $\Sigma$ is given by
\begin{equation*}
    H_\Sigma = \mathcal{H} \circ \eta = \h \circ \nu.
\end{equation*}

In terms of the profile curve of $\Sigma$, we treat next the resulting ODE as the nonlinear autonomous system \eqref{mainsystem}, carrying out a qualitative study of its solutions through a phase space analysis. We also prove in Proposition \ref{propsurfaceaxis} that, up to rotations around the $e_3$-axis, there exists at most a unique helicoidal $\mathcal{H}$-surface in $\R^3$ that intersects its rotation axis.

Building on the ideas of Theorem 4.1 of \cite{ROTATIONAL}, we study in Section \ref{sectionpositive} the helicoidal $\mathcal{H}$-surfaces in cases where $\mathcal{H} \in C^1(\S^2)$ is positive, rotationally invariant and even, i.e. $\mathcal{H}(-x) = \mathcal{H}(x)>0$ for all $x \in \S^2$. Different from the rotational case, we also need to suppose that the function $\h(t)$ related to $\mathcal{H}$ is increasing for $t \in [0,1]$. Then, we arrive at the classification Theorem \ref{teoprincipal}.

Finally, in Section \ref{sectionexamples}, we provide examples of helicoidal $\mathcal{H}$-surfaces in cases where $\mathcal{H}$ vanishes at some point. The biggest problem to extend the orbits behavior of the examples to general choices of $\mathcal{H}$ is not knowing exactly how the phase space looks like. It is important to observe that, even though these examples are specific, similar functions $\mathcal{H}$ can lead to similar behaviors of the resulting helicoidal $\mathcal{H}$-surfaces. Thus, attempting to generalize these models to various choices of $\mathcal{H}$ could inspire future research in this area.


\section{Phase space analysis of helicoidal \texorpdfstring{$\mathcal{H}$}{H}-surfaces}\label{section1}

Let $\Sigma$ be an immersed, oriented surface in $\R^3$ and $\eta : \Sigma \to \S^2$ its Gauss map. Given a function $\mathcal{H} \in C^1(\S^2)$, we will say that $\Sigma$ is a \textit{surface of prescribed mean curvature} $\mathcal{H}$, or simply, \textit{$\mathcal{H}$-surface}, if its mean curvature $H_\Sigma$ is given by
\begin{equation}\label{prescribedHsurface}
    H_\Sigma = \mathcal{H} \circ \eta.
\end{equation}

When $\Sigma$ is given as a graph $x_{3} = u(x_1,x_2)$, then \eqref{prescribedHsurface} is written as the elliptic, second order quasilinear PDE
\begin{equation}\label{ellipticPDE}
    \mathrm{div}\left( \frac{\mathrm{D}u}{\sqrt{1+|\mathrm{D}u|^2}} \right) = 2\mathcal{H}(\mathrm{Z}_u), \quad \mathrm{Z}_u := \frac{(-\mathrm{D}u,1)}{\sqrt{1+|\mathrm{D}u|^2}},
\end{equation}
where $\mathrm{div}$ and $\mathrm{D}$ denote respectively the divergence and gradient operators on $\R^2$ and $\mathrm{Z}_u$ is the unit normal of the graph.

\begin{definition}\label{defhelicoidal}
    A \textit{helicoidal surface} $\Sigma$ in $\R^3$ is an immersed, oriented surface obtained as the orbit of a regular planar curve parametrized by arc-length
    $$\alpha(s) = (x(s),0,z(s)) : I \subset \R \to \R^{3}$$
    under a helicoidal motion, that is, up to isometries of $\R^3$, there exists a constant $c_0 \neq 0$ such that 
    \begin{equation}\label{parametrization}
        \Psi(s,\theta) = \left(x(s)\cos{\theta} , x(s)\sin{\theta} , z(s) + c_0\theta \right) : I \times \R \to \R^3
    \end{equation}
    is a parametrization of $\Sigma$.
\end{definition}

Up to a change of orientation, the Gauss map of a helicoidal surface $\Sigma$ is given by 
\begin{equation}\label{normalhelicoidal}
    \eta(s,\theta) = \frac{\left( c_0 \sin{\theta} x'(s) - \cos{\theta}z'(s)x(s), -\sin{\theta}x(s)z'(s) - c_0\cos{\theta}x'(s), x(s)x'(s) \right)}{\sqrt{c_0^2x'(s)^2 + x(s)^2}}. 
\end{equation}

Therefore, the \textit{angle function} $\nu := \langle \eta, e_3 \rangle$ of $\Sigma$ is given by
\begin{equation}\label{anglefunction}
    \nu(s) = \nu(s,\theta) = \frac{x(s)x'(s)}{\sqrt{c_0^2x'(s)^2 + x(s)^2}}.
\end{equation}

Note that $\|\Psi_s(s,\theta) \wedge \Psi_\theta(s,\theta)\| = 0$ if, and only if, $x(s) = x'(s) = 0$. So, we assume that $x'(s) \neq 0$ for $s \in I$ such that $x(s) = 0$ to $\Sigma$ be a regular surface.

\begin{remark}\label{remarkang1}
    In a helicoidal surface $\Sigma$, its angle function satisfies $|\nu|<1$. Indeed, for $s \in I$ such that $x'(s)=0$, we directly obtain $\nu(s) = 0 < 1$. For $s \in I$ such that $x'(s) \neq 0$, since $|x'(s)|\leq 1$ and $c_0\neq0$, we have
    $$|\nu(s)| = \frac{|x(s)||x'(s)|}{\sqrt{c_0^2x'(s)^2 + x(s)^2}} \leq \frac{|x(s)|}{\sqrt{c_0^2x'(s)^2 + x(s)^2}} < 1.$$ \qed
\end{remark}

The coefficients of the first fundamental form of $\Sigma$ are
\begin{align*}
    E &= \langle \Psi_s, \Psi_s \rangle = \cos^2(\theta)x'(s)^2 + \sin^2(\theta)x'(s)^2 + z'(s)^2 = 1, \\
    F &= \langle \Psi_s, \Psi_\theta \rangle = -\sin\theta \cos\theta x(s)x'(s) + \sin \theta\cos\theta x(s)x'(s) + c_0z'(s) = c_0z'(s), \\
    G &= \langle \Psi_\theta, \Psi_\theta \rangle = \sin^2(\theta) x^2(s) + \cos^2(\theta) x^2(s) + c_0^2 = x(s)^2 + c_0^2,
\end{align*}

and its coefficients of the second fundamental form are

\begin{align*}
    e &= \langle \Psi_{ss} , \eta \rangle = \frac{x(s)\left( z''(s)x'(s) - z'(s)x''(s)\right)}{\sqrt{c_0^2x'(s)^2 + x(s)^2}}, \\
    f &= \langle \Psi_{s\theta} , \eta \rangle = -\frac{c_0x'(s)^2}{\sqrt{c_0^2x'(s)^2 + x(s)^2}},\\
    g &= \langle \Psi_{\theta\theta} , \eta \rangle = \frac{x(s)^2z'(s)}{\sqrt{c_0^2x'(s)^2 + x(s)^2}}.
\end{align*}

Therefore, the mean curvature of a helicoidal surface $\Sigma$ is given by

\begin{align}\label{meancurvaturehelicoidal}
    H_{\Sigma} &= \frac{eG - 2fF + gE}{2(EG-F^2)} = \frac{x(z''x'-z'x'')(x^2+c_0^2) + 2c_0^2x'^2z'+x^2z'}{2\sqrt{c_0^2x'^2 + x^2}\left( x^2 + c_0^2-c_0^2z'^2\right)} \nonumber \\
    &= \frac{(x'z''-z'x'')(x^3+c_0^2x) + z'(2c_0^2x'^2+x^2)}{2(c_0^2x'^2+x^2)^{\frac{3}{2}}}.
\end{align}

Since the mean curvature of a helicoidal surface does not depend on the rotational parameter $\theta$, let $\mathcal{H} \in C^1(\S^n)$ be a rotationally symmetric function, that is, $\mathcal{H}(x) = \h(\langle x,v \rangle)$, for some $v \in \S^2$ and $\h \in C^1([-1,1])$. Up to an Euclidean change of coordinates, we can assume that $v = e_3$, and then $\mathcal{H}(x)=\h(\langle x,e_3 \rangle)$. 

Let now $\Sigma$ be a helicoidal $\mathcal{H}$-surface in $\R^3$. From \eqref{prescribedHsurface}, the mean curvature of $\Sigma$ is given by
\begin{equation}\label{relationHh}
    H_\Sigma = \mathcal{H} \circ \eta = \h \circ \nu.
\end{equation}

Therefore, from \eqref{meancurvaturehelicoidal}, the profile curve $\alpha(s)$ of $\Sigma$ satisfies
\begin{equation}\label{meancurvatureedo}
    2\h\left(\frac{xx'}{\sqrt{c_0^2x'^2 + x^2}}\right) = \frac{(x'z''-z'x'')(x^3+c_0^2x) + z'(2c_0^2x'^2+x^2)}{(c_0^2x'^2+x^2)^{\frac{3}{2}}}.
\end{equation}

Taking into account that $x'^2+z'^2=1$, we can denote $z'=\varepsilon\sqrt{1-x'^2}$, where $\varepsilon = \sign(z')$. Thus,
\begin{equation}\label{formz''}
    z'' = -\varepsilon\frac{x'x''}{\sqrt{1-x'^2}},
\end{equation}
on every subinterval $J \subset I$ where $z'(s)\neq 0$, for all $s \in J$. So, from \eqref{meancurvatureedo}, we obtain that $x(s)$ is a solution to the autonomous second
order ordinary differential equation
\begin{align}
    2(c_0^2x'^2+x^2)^{\frac{3}{2}}\h\left(\frac{xx'}{\sqrt{c_0^2x'^2 + x^2}}\right) &= -\varepsilon x''\left( \frac{x'^2}{\sqrt{1-x'^2}} + \sqrt{1-x'^2} \right)(x^3+c_0^2x) \nonumber \\
    &\phantom{==} + \varepsilon \sqrt{1-x'^2}(2c_0^2x'^2+x^2) \nonumber \\
    &= -\varepsilon\left(\frac{x''}{\sqrt{1-x'^2}}(x^3+c_0^2x) - \sqrt{1-x'^2}(2c_0^2x'^2+x^2) \right) \nonumber \\
    &= -\varepsilon\left( \frac{x''(x^3+c_0^2x) - (1-x'^2)(2c_0^2x'^2+x^2)}{\sqrt{1-x'^2}} \right), \label{EDOx''1}
\end{align}
that is, when $x \neq 0$,
\begin{equation}\label{EDOx''}
    x'' = \frac{(1-x'^2)(x^2+2c_0^2x'^2)-2\varepsilon\h\left(\frac{xx'}{\sqrt{c_0^2x'^2 + x^2}}\right)\sqrt{1-x'^2}(c_0^2x'^2+x^2)^{\frac{3}{2}}}{x^3+c_0^2x}.
\end{equation}

Denoting $y=x'$, the ODE \eqref{EDOx''} transforms into the first order autonomous system
\begin{equation}\label{mainsystem}
    \left( \begin{array}{c} x \\ y \end{array}\right)' = \left( \begin{array}{c} y \\ \frac{(1-y^2)(x^2+2c_0^2y^2)-2\varepsilon\h\left(\frac{xy}{\sqrt{c_0^2y^2 + x^2}}\right)\sqrt{1-y^2}(c_0^2y^2+x^2)^{\frac{3}{2}}}{x^3+c_0^2x} \end{array}\right) =: F(x,y).
\end{equation}

The phase space of the system is
$$\Theta_{\varepsilon} := \R\smallsetminus\{0\} \times (-1,1)$$
with coordinates $(x, y)$ denoting, respectively, the distance to the rotation axis and its variation.

If the profile curve $\alpha(s) = (x(s),0,z(s))$ intersects the $e_3$-axis, that is, $x(s)=0$ for some $s \in I$, we obtain from \eqref{EDOx''1} that
$$\varepsilon\h(0) = \frac{2(1-x'(s)^2)c_0^2x'(s)^2}{2\sqrt{1-x'(s)^2}|c_0|^3|x'(s)|^3} = \frac{\sqrt{1-x'(s)^2}}{|c_0||x'(s)|} \quad \Longrightarrow \quad \h(0)^2 = \frac{1-x'(s)^2}{c_0^2x'(s)^2},$$
that is, $\varepsilon \h(0) \geq 0$ and $\alpha(s)$ satisfies
\begin{equation}\label{eqtocaeixo}
    x'(s) = \pm \frac{1}{\sqrt{1+c_0^2\h(0)^2}}.
\end{equation}

In particular, its associated orbit in $\Theta_\varepsilon$, such that $\varepsilon\h(0) \geq 0$, must intersect the $x=0$ axis at the points 
$$p_\varepsilon := \left( 0, \frac{1}{\sqrt{1+c_0^2\mathfrak{h}(0)^2}}\right) \quad \text{or} \quad -p_\varepsilon = \left( 0, -\frac{1}{\sqrt{1+c_0^2\mathfrak{h}(0)^2}}\right).$$

Moreover, observe that no orbit in $\Theta_\varepsilon$ can have a limit point of the form $(0,y)$, with $y \neq \pm\frac{1}{\sqrt{1+c_0^2\mathfrak{h}(0)^2}}$. Indeed, since $\h \in C^1([-1,1])$ and the equation \eqref{EDOx''1} is continuous around $\Theta_\varepsilon$, we can take $x(s) \to 0$ and obtain $y(s) \to \pm\frac{1}{\sqrt{1+c_0^2\mathfrak{h}(0)^2}}$.

\begin{remark}\label{obsnegative}
    Let $\gamma_1(s)=(x(s), y(s))$ be any solution to \eqref{mainsystem}. Observe that $\gamma_2(s)=(-x(s),-y(s))$ is also a solution to the system. This means that the phase space $\Theta_\varepsilon$ is symmetric with respect to the origin. Thus, without loss of generality, it is enough to analyse the case $x(s) > 0$ and consider the phase space 
    $$\Theta_{\varepsilon} = (0,\infty) \times (-1,1).$$

    If the profile curve intersects the $e_3$-axis, its associated orbit $(x(s),y(s))$ in $\Theta_\varepsilon$ must pass through the point $\pm p_\varepsilon$ at some $s_0 \in I$ ($x(s_0)=0$). Since $x'(s_0) \neq 0$, let $\delta > 0$ such that, for $s \in (s_0,s_0+\delta)$ or $s \in (s_0-\delta,s_0)$, we would have $x(s)<0$. Note that we can analyse the orbit behavior from $\mp p_\varepsilon$, considering the symmetry, and obtain $x(s)>0$ for $s \in (s_0 - \delta, s_0 + \delta) \smallsetminus \{s_0\}$. 

    Moreover, let $\alpha_i(s) = (x_i(s),0,z_i(s))$, $i=1,2$, be the profile curve of the helicoidal $\mathcal{H}$-surface $\Sigma_i$ associated to $\gamma_i(s)$ and let $\Psi_i(s,\theta)$ be the parametrization of $\Sigma_i$, as in \eqref{parametrization}. Letting $z_0 := z_1(0)$ and considering, up to vertical translations, that $z_2(0) = z_0 + c_0\pi$, we obtain that
    $$\Psi_2(s,\theta) = \Psi_1(s,\theta + \pi),$$
    that is, up to a change of coordinates that corresponds to a rotation by $\pi$ around the $e_3$-axis, $\Sigma_1 = \Sigma_2$.

    Indeed, note that $x_2(s) = -x_1(s) = -x(s)$ and 
    $$z_2(s) = \int_0^s \varepsilon \sqrt{1-(-y(t))^2}dt + z_2(0) = \int_0^s \varepsilon \sqrt{1-(y(t))^2}dt + z_0 + c_0\pi = z_1(s) + c_0\pi.$$
    Therefore,
    \begin{align*}
        \Psi_2(s,\theta) &= \left(x_2(s)\cos{\theta},\, x_2(s)\sin{\theta},\, z_2(s) + c_0\theta \right) \\
        &= \left(-x(s)\cos{\theta},\,-x(s)\sin{\theta},\, z_1(s) + c_0\pi + c_0\theta \right) \\
        &= \left(x(s)\cos{(\theta+\pi)},\,x(s)\sin{(\theta+\pi)},\, z_1(s) + c_0(\pi +\theta) \right) \\
        &= \Psi_1(s,\theta+\pi).
    \end{align*} \qed
\end{remark}

\begin{proposition}\label{propsurfaceaxis}
    There exists at most a unique (up to rotations around the $e_3$-axis) helicoidal $\mathcal{H}$-surface in $\R^3$ that intersects its rotation axis.
\end{proposition}

\begin{proof}
    Let $\Sigma_1$ and $\Sigma_2$ be helicoidal $\mathcal{H}$-surfaces that meet the $e_3$-axis. Then, this axis is contained in both surfaces. Let $\Psi_1(s,\theta)$ and $\Psi_2(s,\theta)$ be the parametrization of $\Sigma_1$ and $\Sigma_2$, respectively, as in \eqref{parametrization}, where, up to rotations around the $e_3$-axis and a change of coordinates, $\alpha_1(s) = (x_1(s),0,z_1(s))$ and $\alpha_2(s) = (x_2(s),0,z_2(s))$ are the profile curves of $\Sigma_1$ and $\Sigma_2$, respectively, in the plane $\Pi = \{y=0\}$, such that $\Psi_i(s,0) = \alpha_i(s)$, $i=1,2$, and $\alpha_1(0) = \alpha_2(0) = p_0$, for some $p_0 = (0,0,z_0) \in \Sigma_1 \cap \Sigma_2$.

    Consider $\delta_i := \sign(x_i'(0))$, $i=1,2$. By \eqref{normalhelicoidal}, we obtain that the Gauss map of $\Sigma_i$ in $p_0$ is given by
    $$\eta(0,0) = \frac{\left(c_0\sin(0)x_i'(0),-c_0\cos(0)x_i'(0),0\right)}{|c_0||x_i'(0)|} = -\sign(c_0)\left(0,\delta_i,0 \right).$$
    
    Let $\varepsilon \in \{-1,1\}$ such that $\varepsilon\h(0) \geq 0$, where $\h$ is given by \eqref{relationHh} in terms of $\mathcal{H}$. By \eqref{eqtocaeixo}, we have 
    $$|x_1'(0)| = |x_2'(0)| = \frac{1}{\sqrt{1+c_0^2\mathfrak{h}(0)^2}},$$
    that is, the orbit associated to the profile curve of $\Sigma_i$ passes through the point $\delta_i p_\varepsilon \in \Theta_\varepsilon$. If $\delta_1 \neq \delta_2$, we can rotate $\Sigma_2$ by $\pi$ around the $e_3$-axis and, then, by Remark \ref{obsnegative}, obtain that $\Sigma_1$ and $\Sigma_2$ have the same unit normal in $p_0$.

    On a sufficiently small ball $B_\delta(p_0) \subset \Pi$, we can see $\Sigma_1$ and $\Sigma_2$ as graphs of differential functions $u_1(x,z)$ and $u_2(x,z)$, respectively, with
    $$u_1(0,z) = u_2(0,z) = 0, \quad \forall z \in (z_0 - \delta, z_0 + \delta),$$
    where the difference $v:=u_1-u_2$ is a pseudoanalytic function, since $u_1$ and $u_2$ are solutions to the same elliptic PDE \eqref{ellipticPDE} (\cite[p. 469]{HARTMAN}). 
    
    If there exists some $(x_0,z_0) \in B_\delta(p_0)$, with $x_0 \neq 0$, such that $u_1(x_0,z_0) = u_2(x_0,z_0)$, consider $s_0, \theta_0 \in \R$ such that $\Psi_1(s_0,\theta_0) = u_1(x_0,z_0) = \Psi_2(s_0,\theta_0)$. This means that $\Psi_1(s_0,\theta)=\Psi_2(s_0,\theta)$, for all $\theta \in \R$, that is, $u_1 = u_2$ along the curve $\gamma_{s_0}(\theta) := \Psi_1(s_0,\theta)$ in $\Sigma_1$ and $\Sigma_2$. In case there exists a set of accumulated curves $\gamma_{s_0}(\theta)$, we obtain that $u_1\equiv u_2$ by pseudoanalyticity, i.e. $\Sigma_1 = \Sigma_2$. Therefore, the curves $\gamma_{s_0}$ are isolated and we can get a sufficiently small ball $B_\epsilon(p_0)$ such that $u_1\leq u_2$ in 
    $$B_\epsilon(p_0) \cap \{ x\leq 0 \}.$$
    
    By \eqref{ellipticPDE}, the functions $u_1$ and $u_2$ satisfy the same absolutely elliptic partial differential equation. It follows from \cite[Corollary 4.6]{HOPF} that $u_1 \equiv u_2$ and, therefore, up to rotations around the $e_3$-axis, $\Sigma_1 = \Sigma_2$.
\end{proof}

A point $e_0 =(x_0,y_0) \in \Theta_\varepsilon$ is a \textit{equilibrium} of \eqref{mainsystem} if $F(x_0,y_0) = 0$, i.e. $y_0=0$ and 
$$x_0^2-2\varepsilon\h\left(0\right)x_0^3 = 0 \quad \Longrightarrow \quad x_0 = \frac{1}{2\varepsilon\h(0)}.$$

Thus, if $\varepsilon\h(0) > 0$, the equilibrium in $\Theta_\varepsilon$ is the point 
$$e_0 = \left( \frac{1}{2\varepsilon\h(0)}, 0 \right)$$
and it corresponds to the case where $\Sigma$ is a right circular cylinder $\S^1(1/2\h(0)) \times \R$ in $\R^3$ of constant mean curvature $\h(0)$ and vertical rulings. Note that the cylinder does not depend on the pitch $c_0$. If $\varepsilon\h(0) \leq 0$, there is not a equilibrium in $\Theta_\varepsilon$.

The orbits $(x(s), y(s))$ provide a foliation by regular proper $C^1$ curves of $\Theta_\varepsilon$, or of $\Theta_\varepsilon \smallsetminus \{e_0\}$, in case $e_0$ exists. Since $\h \in C^1([-1,1])$, the uniqueness of the initial value problem for \eqref{mainsystem} implies that if an orbit $(x(s), y(s))$ converges to $e_0$, the value of the parameter $s$ goes to $\pm\infty$.

The nullcline which corresponds to $x'=0$ is given by the axis $y=0$ in $\Theta_\varepsilon$. Therefore, $x'(s) > 0$ if $y(s) > 0$ and $x'(s) < 0$ if $y(s) < 0$. Denote by $\Gamma_\varepsilon$ the nullcline in $\Theta_\varepsilon$ which corresponds to $y'=0$, that is, the set of points $(x,y) \in \Theta_\varepsilon$ such that
\begin{equation}\label{formgamma}
    \frac{\sqrt{1-y^2}(x^2+2c_0^2y^2)}{2(x^2+c_0^2y^2)^{\frac{3}{2}}}=\varepsilon\h\left(\frac{xy}{\sqrt{c_0^2y^2 + x^2}}\right).
\end{equation}
Alternatively, we can describe $\Gamma_\varepsilon$ as the implicit equation $F_\varepsilon(x,y) = 0$, where $(x,y) \in \Theta_\varepsilon$ and $F_\varepsilon$ is the differential function given by
\begin{equation}\label{formgammaF}
    F_\varepsilon(x,y):= 2\varepsilon\h\left(\frac{xy}{\sqrt{c_0^2y^2 + x^2}}\right)(x^2+c_0^2y^2)^{\frac{3}{2}} - (x^2+2c_0^2y^2)\sqrt{1-y^2}.
\end{equation}

Note that $\Gamma_\varepsilon$ might be empty, for example, in the case $\h\leq 0$ and $\varepsilon=1$. Furthermore, observe that, if $\Gamma_\varepsilon$ exists, it must intersect the $x=0$ axis at the points $\pm p_\varepsilon$ and intersect the $y=0$ axis only at the equilibrium $e_0$.

The values $s \in J$ where the profile curve $\alpha(s)=(x(s),0,z(s))$ of $\Sigma$ has zero geodesic curvature $\kappa_\alpha$ are those where its associated orbit $(x(s),y(s))$ belongs to $\Gamma_\varepsilon$. Indeed, using \eqref{formz''}, we obtain that for each $s \in J$,
\begin{align*}
    \kappa_\alpha(s) &= x'(s)z''(s) - x''(s)z'(s) \\
        &= -\varepsilon\frac{x'(s)^2x''(s)}{\sqrt{1-x'(s)^2}} - \varepsilon x''(s)\sqrt{1-x'(s)^2} \\
        &= -\varepsilon y'(s)\left( \frac{y(s)^2}{\sqrt{1-y(s)^2}}+\sqrt{1-y(s)^2}\right) \\
        &= -\frac{\varepsilon y'(s)}{\sqrt{1-y(s)^2}},
\end{align*}
that is, $\kappa_\alpha(s) = 0$ if, and only if, $y'(s) = 0$.

\begin{remark}\label{convergeborda}
    Given $\gamma(s) = (x(s),y(s)) \in \Theta_\varepsilon$ an orbit, if $\gamma(s) \mapsto (x_0, \pm 1)$, for some $x_0 \geq 0$, then there exists $s_0 \in \R$ such that $\gamma(s_0) = (x_0,\pm 1)$. Indeed, if $\gamma(s) \mapsto (x_0, \pm 1)$ only when $s \to \infty$, we can get a sequence \{$x(s_n):n\in\mathbb{N}\}$, monotonous for $n\geq n_0$ sufficiently large, such that $x(s_n) \to x_0$. It follows that $x'(s_n) = x(s_{n+1}) - x(s_n) \to 0$, which is a contradiction, because $x'(s)=y(s) \to \pm 1$. \qed
\end{remark}


\section{Positive mean curvature helicoidal surfaces}\label{sectionpositive}

First, observe that if $\mathcal{H}\equiv H_0 > 0$, that is, the resulting $\mathcal{H}$-surface has constant mean curvature $H_0$, we get that the curve $\Gamma_1$ is given by the points $(x,y) \in \Theta_1$ such that
$$x^2 + 2c_0^2y^2 = \frac{2H_0(x^2+c_0^2y^2)^{\frac{3}{2}}}{\sqrt{1-y^2}} \quad \Longrightarrow \quad H_0 = \frac{(x^2+2c_0^2y^2)\sqrt{1-y^2}}{2(x^2+c_0^2y^2)^{\frac{3}{2}}}.$$

So, $\Gamma_1$ corresponds to the level curve $f_1 \equiv H_0$, where
\begin{equation}\label{functionepsilon}
    f_\varepsilon(x,y) := \frac{\varepsilon(x^2+2c_0^2y^2)\sqrt{1-y^2}}{2(x^2+c_0^2y^2)^{\frac{3}{2}}}. \text{ (See Fig. \ref{curvnivel}).}
\end{equation}

\begin{figure}[h]
    \centering
    \includegraphics[scale=0.5]{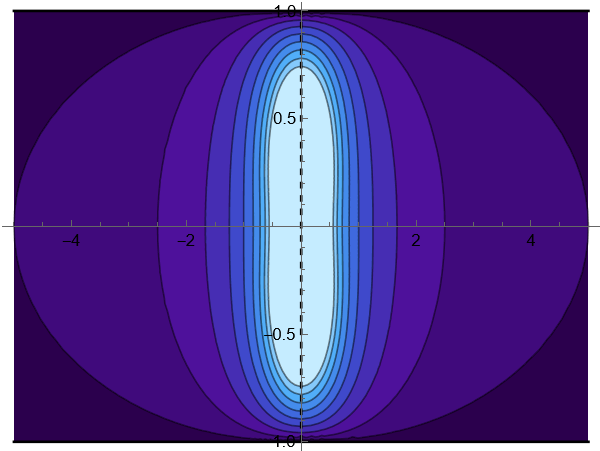}
    \caption{Level curves of $f_\varepsilon$, for different values of $H_0$.}
    \label{curvnivel}
\end{figure}

\begin{remark}
    The intersection points of the level curve $f_\varepsilon \equiv H_0$ with the $x=0$ axis are $p_\varepsilon$ and $-p_\varepsilon$. Indeed, note that for $x=0$ we have
    $$H_0 = f_\varepsilon(0,y) = \frac{\varepsilon\sqrt{1-y^2}}{|c_0||y|} \quad \Longrightarrow \quad y = \pm \frac{1}{\sqrt{1+c_0^2H_0^2}},$$
    and $\varepsilon H_0 > 0$. \qed
\end{remark}

We would like to get similar results from Theorem 4.1 of \cite{ROTATIONAL} for helicoidal $\mathcal{H}$-surfaces in case that $\mathcal{H} \in C^1(\S^2)$ is positive, rotationally invariant and even, i.e. $\mathcal{H}(-x) = \mathcal{H}(x)>0$ for all $x \in \S^2$. In terms of the function $\h$ associated to $\mathcal{H}$, this means that $\h(-y)=\h(y) > 0$ for all $y \in [-1,1]$. We need to add one more hypothesis, different from the rotational case, to guarantee the behavior of $\Gamma_1$: we suppose that the function $\h(t)$ is increasing for $t \in [0,1]$ and, so, decreasing for $t \in [-1,0]$. It is possible to get similar results for $\h(t)$ decreasing for $t \in [0,1]$ and increasing for $t \in [-1,0]$.


In this case, the curve $\Gamma_{-1}$ does not exist and $\Gamma_1$ is given by the implicit equation $F_1(x,y) = 0$, where $F_1$ is defined in \eqref{formgammaF}.

\begin{proposition}\label{propgamma1}
    If $\h(t) \in C^1([-1,1])$ is positive, even and increasing for $t\in [0,1]$, the curve $\Gamma_1$ is regular, compact and connected.
\end{proposition}

\begin{proof}
    First, note that we can describe $\Gamma_1$ as the set
    $$\left\{(x,y) \in \Theta_1 : f_1(x,y)=\h\left(\frac{xy}{\sqrt{c_0^2y^2 + x^2}}\right) \right\},$$
    where $f_1$ is the smooth function defined in \eqref{functionepsilon}. Moreover, observe that $\Gamma_1$ is symmetric with respect to the $y=0$ axis, by reason of $\h$ is even.
    
    Since $[-1,1]$ is compact and $\h$ is continuous, there are minimum $H_m$ and maximum $H_M$ values of the $\h$ function, that is, $0<H_m\leq \h \leq H_M$. Thus, considering that $\pm p_1 \in \Gamma_1$, we get that $\Gamma_1$ is compact because it is a closed set contained in the compact region
    $$D:=\{ H_m \leq f_1(x,y) \leq H_M \} \cap  \{ x \geq 0 \}. \text{  (See Fig. \ref{regionfunction}).}$$
    
    \begin{figure}[h]
        \centering
        \includegraphics[scale=0.5]{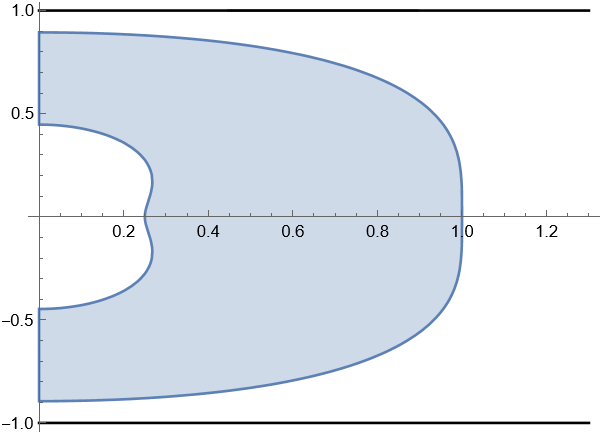}
        \caption{A sketch of the $D$ region.}
        \label{regionfunction}
    \end{figure}
    
    Consider the new variables
    $$u=x^2 \quad \text{and} \quad v=\frac{xy}{\sqrt{x^2+c_0^2y^2}}.$$
    Note that the application $G: \R^2 \to \R^2$ given by $G(x,y) = \left(x^2,\frac{xy}{\sqrt{x^2+c_0^2y^2}} \right)$ is a diffeomorphism from $\Theta_1$ to $G(\Theta_1)$ that preserves orientation, because the Jacobian determinant
    $$\J G(x,y) = \left|\begin{array}{cc}
        2x & 0 \\
        \frac{c_0^2y^3}{(x^2+c_0^2y^2)^{\frac{3}{2}}} & \frac{x^3}{(x^2+c_0^2y^2)^{\frac{3}{2}}}
    \end{array} \right| = \frac{2x^4}{(x^2+c_0^2y^2)^{\frac{3}{2}}} > 0$$ 
    is positive, for all $(x,y) \in \Theta_1$. Furthermore, we obtain
    $$G(\Theta_1) = \left\{ (u,v) \in \R^2 : u > 0, |v| < \sqrt{\frac{u}{u+c_0^2}} \right\}.$$ 
    
    Considering the changing of variables, we obtain
    $$v^2=\frac{x^2y^2}{x^2+c_0^2y^2} = \frac{uy^2}{u+c_0^2y^2} \quad \Longrightarrow \quad y^2 = \frac{v^2u}{u-c_0^2v^2}.$$
    
    It is important to observe that $u-c_0^2v^2 > 0$ in $G(\Theta_1)$, because
    $$u-c_0^2v^2 = x^2 - \frac{c_0^2x^2y^2}{x^2+c_0^2y^2}= \frac{x^4}{x^2+c_0^2y^2}>0,$$
    for all $(x,y) \in \Theta_1$. Therefore, by \eqref{formgammaF}, we get that $\Gamma_1$ is characterized by the implicit equation $\overline{F}_1(u,v)=0$, $(u,v) \in G(\Theta_1)$, where 
    \begin{align*}
        \overline{F}_1(u,v) &= 2\h(v)\left( u + \frac{c_0^2v^2u}{u-c_0^2v^2} \right)^{\frac{3}{2}} - \left( u + \frac{2c_0^2v^2u}{u-c_0^2v^2} \right)\sqrt{1-\frac{v^2u}{u-c_0^2v^2}} \\
        &= \frac{1}{\left( u-c_0^2v^2 \right)^\frac{3}{2}}\left( 2\h(v)u^3-(u^2+c_0^2v^2u)\sqrt{u(1-v^2)-c_0^2v^2} \right) \\
        &= \frac{u}{\left( u-c_0^2v^2 \right)^\frac{3}{2}}\left( 2\h(v)u^2-(u+c_0^2v^2)\sqrt{u(1-v^2)-c_0^2v^2} \right).
    \end{align*}
    To simplify, we can also describe $\Gamma_1$ as the implicit equation $F_2(u,v) = 0$, $(u,v) \in G(\Theta_1)$, where $F_2$ is the $C^1$ function given by
    $$F_2(u,v) := 2\h(v)u^2-(u+c_0^2v^2)\sqrt{u(1-v^2)-c_0^2v^2}.$$
    
    Let us prove that $0$ is a regular value of $F_2$. Note that
    \begin{align}
        \frac{\partial F_2}{\partial u}(u,v) &= \frac{c_0^2v^2(v^2+1) + 3u(v^2-1) + 8\h(v)u\sqrt{u(1-v^2)-c_0^2v^2}}{2\sqrt{u(1-v^2)-c_0^2v^2}}, \nonumber \\
        \frac{\partial F_2}{\partial v}(u,v) &= \frac{v\left( u^2 + 3c_0^4v^2 + c_0^2u(3v^2-1) \right)}{\sqrt{u(1-v^2)-c_0^2v^2}} + 2u^2\h'(v). \label{equ2}
    \end{align}
    
    Let $(u,v) \in G(\Theta_1)$ be a point of $\Gamma_1$ such that $\frac{\partial F_2}{\partial u}(u,v) = 0$. Then, a computation shows that
    \begin{equation}\label{equ1}
        u = \frac{c_0^2v^2\left(3v^2-1+\sqrt{\left( 3v^2 - \frac{11}{3} \right)^2 + \frac{32}{9}}\right)}{2(1-v^2)}.
    \end{equation}
    Observe that the equilibrium $\left(\frac{1}{4\mathfrak{h}(0)^2},0\right)$, the only point of the $v=0$ axis contained in $\Gamma_1$, does not satisfy the above equation, so we can assume that $v \neq 0$.
    
    By replacing \eqref{equ1} in \eqref{equ2}, a straightforward computation shows that $\frac{\partial F_2}{\partial v}(u,v) > 0$, if $v > 0$, and $\frac{\partial F_2}{\partial v}(u,v) < 0$, if $v < 0$, since $\h'(v)\geq 0$, for $v > 0$, and $\h'(v)\leq 0$, for $v < 0$, by hypothesis. That means that there does not exist a critical point of $F_2$ contained in $\Gamma_1$, which proves that $0$ is a regular value of $F_2$, that is, the curve $\Gamma_1 = F_2^{-1}(0)$ is regular.
     
    Now, let us prove that $\Gamma_1$ is connected. Let $\Gamma_1^1$ be the connected component of $\Gamma_1$ that passes through the equilibrium $e_0$. Since $\Gamma_1$ is contained in the compact $D$ region and $e_0$ is the only point in the $y=0$ axis that belongs to $\Gamma_1$, we get that $\Gamma_1^1$ has $\pm p_1$ as endpoints. This means that $\Gamma_1$ divides the phase space $\Theta_1$ into two regions $\Lambda_0$ and $\Lambda_\infty$, where the second one is unbounded.

    Suppose that there exists a different connected component $\Gamma_1^2$ of $\Gamma_1$, such that $\Gamma^2_1$ is not contained in the inner domain bounded by some other possible connected component. Then, observe that $\Gamma_1^2$ must be a closed curve contained in $\Lambda_0 \cap \{ y \neq 0 \}$ or $\Lambda_\infty \cap \{ y \neq 0 \}$. Since $\Gamma_1$ is symmetric with respect to the $y=0$ axis, it is sufficient to analyse the case $y>0$. First, consider $\Gamma_1^2 \subset \Lambda_0 \cap \{ y>0 \}$. 
    
    Observe that, for points $(x,0) \in \Theta_1$ in the $y=0$ axis and $\Lambda_0$ region, that is, points such that $x < \frac{1}{2\mathfrak{h}(0)}$, it follows from \eqref{formgammaF} that 
    $$F_1(x,0) = 2\h(0)x^3 - x^2 = x^2\left( 2\h(0)x - 1 \right) < 0.$$
    This means, by regularity of $\Gamma_1 = F_1^{-1}(0)$, that we can take a point $p_0=(x_0,y_0) \in \Theta_1$, $y_0 > 0$, contained in the interior region determined by $\Gamma_1^2$ such that $F_1(x_0,y_0) > 0$, that is, 
    \begin{equation}\label{equu3}
        \h\left( \frac{x_0y_0}{\sqrt{x_0^2 + c_0^2y_0^2}} \right) > \frac{(x_0^2+2c_0^2y_0^2)\sqrt{1-y_0^2}}{2(x_0^2+c_0^2y_0^2)^{\frac{3}{2}}}.
    \end{equation}
    Note that we can also take a point $(x_1,y_0) \in \Theta_1$ with same height of $(x_0,y_0)$ contained in the $\Lambda_0$ region, with $x_1 > x_0$, such that $F_1(x_1,y_0) < 0$. See Figure \ref{esquemaconexo}.

    \begin{figure}
        \centering
        \includegraphics[scale=7.5]{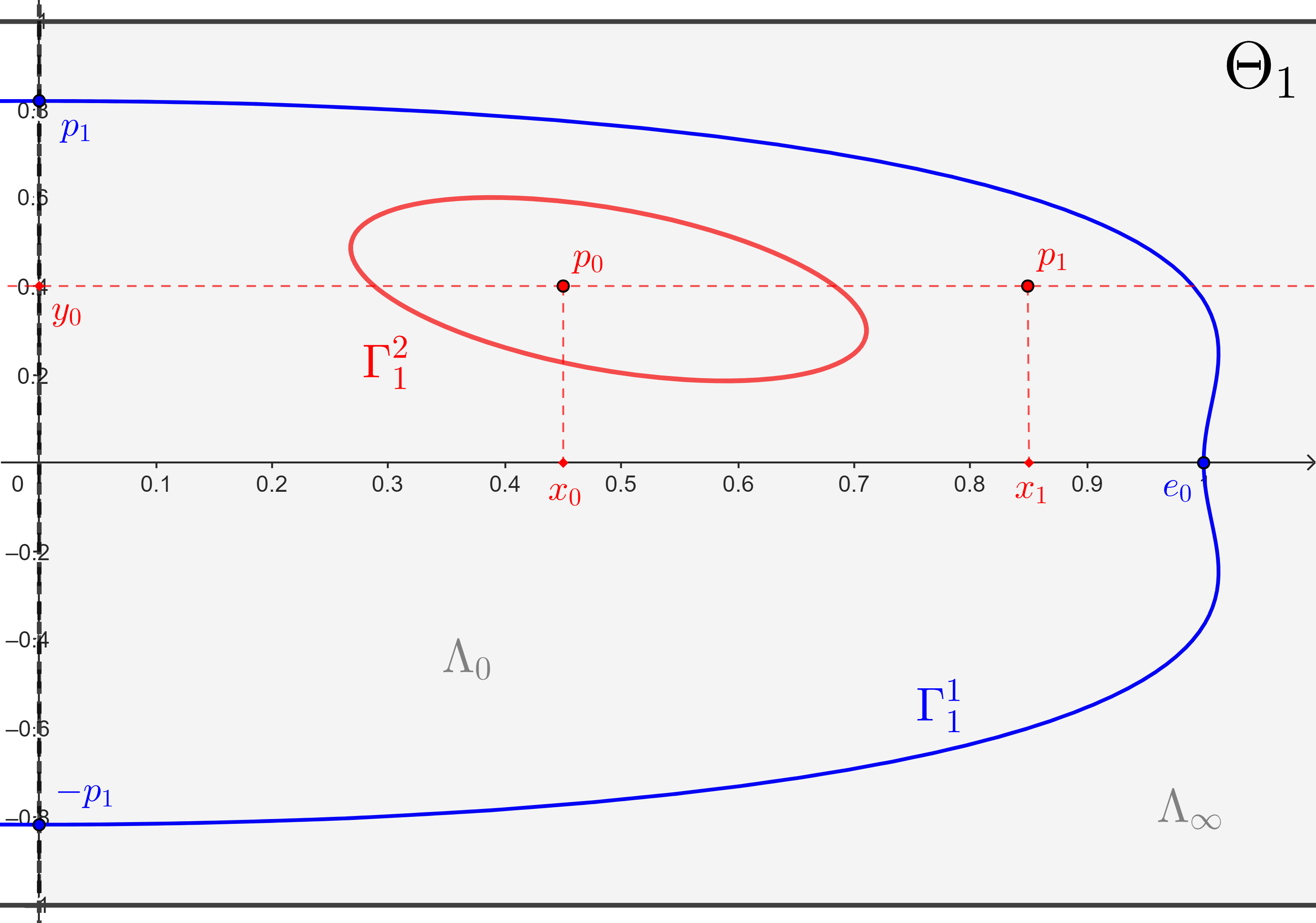}
            \caption{A sketch of the connected components of $\Gamma_1$.}
        \label{esquemaconexo}
    \end{figure}

    Define $\h_0$ and $g_0$ by
    $$\h_0(x) := \h\left( \frac{y_0x}{\sqrt{x^2 + c_0^2y_0^2}} \right) \quad \text{and} \quad g_0(x) := \frac{(x^2+2c_0^2y_0^2)\sqrt{1-y_0^2}}{2(x^2+c_0^2y_0^2)^{\frac{3}{2}}}, \quad x > 0.$$
    Since $\frac{y_0x}{\sqrt{x^2 + c_0^2y_0^2}}$ is increasing for $x>0$ and, by hypothesis, the function $\h(t)$ is increasing for $t>0$, we get that $\h_0$ is an increasing function. Moreover, a computation shows that $g_0$ is strictly decreasing. Hence,
    $$\h_0(x_0) \leq \h_0(x_1) \quad \text{and} \quad g_0(x_0) > g_0(x_1).$$
    From \eqref{equu3} and using that $F_1(x_1,y_0) < 0$, we get that
    $$\h_0(x_0) > g_0(x_0) \quad \text{and} \quad \h_0(x_1) < g_0(x_1).$$
    Consequently,
    $$\h_0(x_0) \leq \h_0(x_1) < g_0(x_1) < g_0(x_0) < \h_0(x_0),$$
    which is contradiction. If $\Gamma_1^2 \subset \Lambda_\infty \cap \{ y>0 \}$, it is possible to get a similar contradiction, taking a point $(x_1,y_0) \in \Lambda_\infty$ such that $x_1 < x_0$.
    
    Therefore, $\Gamma_1^1$ is the only connected component of $\Gamma_1$, that is, $\Gamma_1$ is connected.
\end{proof}

\begin{theorem}\label{teoprincipal}
    Let $\h \in C^1([-1,1])$ be given by \eqref{relationHh} in terms of $\mathcal{H}$, and suppose that $\h$ is positive, even ($\h(t) = \h(-t) > 0$, for all $t \in [-1,1]$) and increasing for $t \in [0,1]$. Then, the following list exhausts, up to vertical translations, all existing helicoidal $\mathcal{H}$-surfaces in $\R^3$ around the $e_3$-axis: 
    \begin{enumerate}
        \item The right circular cylinder $C_\mathcal{H}:= \S^1\left( \frac{1}{2\mathfrak{h}(0)} \right) \times \R$;
        
        \item The orbit under a helicoidal motion of a curve $\alpha(s) = (x(s),0,z(s))$, with $x(s)$ periodic and $x(s_0 + kT) = 0$, for some $s_0, T \in \R$, and for all $k \in \mathbb{N}$, such that $z'(s) > 0$, for all $s \in \R$;
        
        \item The orbit under a helicoidal motion of a closed curve $\alpha(s) = (x(s),0,z(s))$, diffeomorphic to the circle $\S^1$ , such that $x(s) > 0$, for all $s \in \R$;
        
        \item A one-parameter family of helicoidal surfaces with a \textbf{nodoid structure}, that is, the orbit under a helicoidal motion of a curve $\alpha(s) = (x(s),0,z(s))$, such that $x(s) > 0$ is periodic and the curve $\alpha$ is not embedded;
        
        \item A one-parameter family of helicoidal surfaces with a \textbf{unduloid structure}, that is, the orbit under a helicoidal motion of a curve $\alpha(s) = (x(s),0,z(s))$, such that $x(s) > 0$ is periodic and $z'(s)>0$, for all $s \in \R$.
    \end{enumerate}
\end{theorem}

\begin{proof}
    Let $(x(s),y(s))$ be any solution to \eqref{mainsystem}. Taking into account that $(x(-s),-y(-s))$ is also a solution to the system, we obtain that any orbit of the phase space $\Theta_\varepsilon$ is symmetric with respect to the $y=0$ axis. By Proposition \ref{propgamma1}, the curve $\Gamma_1$ in $\Theta_1$ together with $y=0$ divides the phase space $\Theta_1$ into four monotonicity regions $\Lambda_1,\dots,\Lambda_4$, all of them meeting at the equilibrium $e_0$. Moreover, $\Theta_{-1}$ has only two monotonicity regions $\Lambda^+=\Theta_{-1}\cap\{y>0\}$ and $\Lambda^-=\Theta_{-1}\cap\{y<0\}$, because there does not exist $\Gamma_{-1}$ in $\Theta_{-1}$. See Figure \ref{phasespacepositive}.

    \begin{figure}[h]
        \centering
        \includegraphics[scale=0.12]{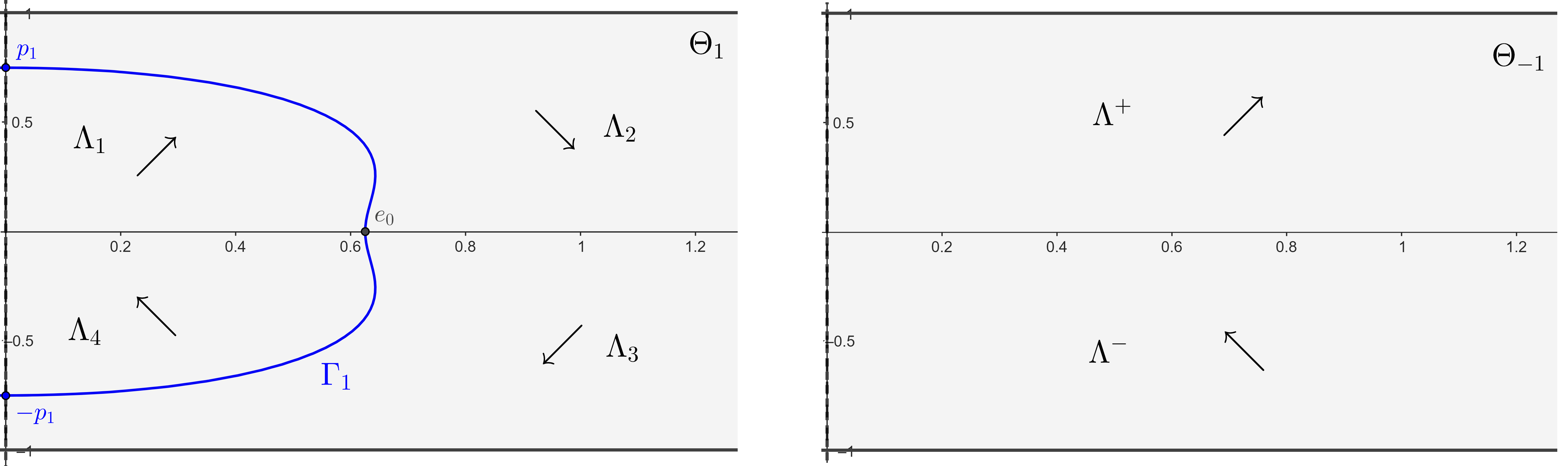}
            \caption{The phase spaces $\Theta_1$ and $\Theta_{-1}$ for $\h(t)$ positive, even and increasing for $t \in [0,1]$.}
        \label{phasespacepositive}
    \end{figure}
    
    By monotonicity properties, any orbit in $\Theta_{-1}$ is given by a horizontal $C^1$ graph $x=g(y)$, with $g(y)=g(-y)>0$ for all $y\in (-\delta,\delta)$, for some $\delta \in (0,1]$, and such that $g$ restricted to $[0,\delta)$ is strictly increasing. 
    
    In case that $g(y) \to \infty$ as $y \to \delta$ for some orbit, let $\alpha(s) = (x(s),0,z(s))$ be its associated profile curve and $r_0 = g(0)$. Consider $\Psi(s,\theta)$ a parametrization for the corresponding helicoidal $\mathcal{H}$-surface $\Sigma$, as in \eqref{parametrization}. If we restrict $\theta \in (-\frac{\pi}{2},\frac{\pi}{2})$, the resulting surface $\Sigma$ would be a bi-graph over the set
    $$\Omega = \R^2 \smallsetminus B_{r_0}(0) \cap \{ x>0 \}.$$
    
    This is impossible, since $\mathcal{H} \geq H_0 > 0$ and it is well know, by the maximum principle comparing with the sphere $\mathbb{S}^2(\frac{1}{2H_0})$, that there do not exist graphs in $\R^3$ over $\Omega$ with mean curvature bounded from below by a positive constant. Therefore, $\delta = 1$ and any orbit in $\Theta_{-1}$ has two limit endpoints of the form $(x_0,\pm 1)$ for some $x_0 > 0$. The resulting $\mathcal{H}$-surface $\Sigma_{-1}$ in $\R^3$ associated to any such orbit is then a surface with boundary parametrized by the curves 
    $$\gamma_1(t)=(x_0\cos t, x_0\sin t, a + c_0 t) \quad \text{and} \quad \gamma_2(t)=(x_0\cos t, x_0\sin t, b + c_0 t), \quad t \in \R,$$
    for some $a<b$. The $z(s)$-coordinate of the profile curve $\alpha(s)$ of $\Sigma_{-1}$ is strictly decreasing and the unit normal to $\Sigma_{-1}$ along $\gamma_1$ and $\gamma_2$ are, respectively,
    $$\eta_1(t) = \frac{(c_0\sin t,-c_0\cos t,x_0)}{\sqrt{c_0^2+x_0^2}} \quad \text{and} \quad \eta_2(t) = \frac{(c_0\sin t,-c_0\cos t,-x_0)}{\sqrt{c_0^2+x_0^2}}.$$
    
    \begin{figure}
        \centering
        \includegraphics[scale=0.5]{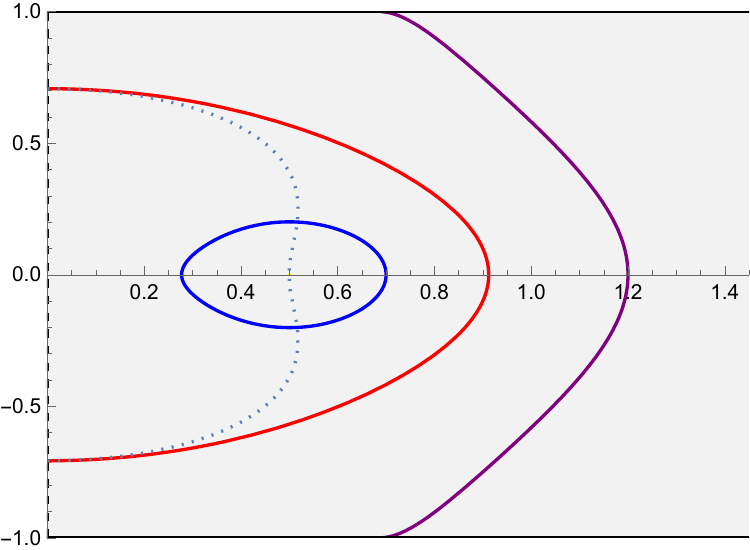}
        \hspace{0.15cm}
        \includegraphics[scale=0.5]{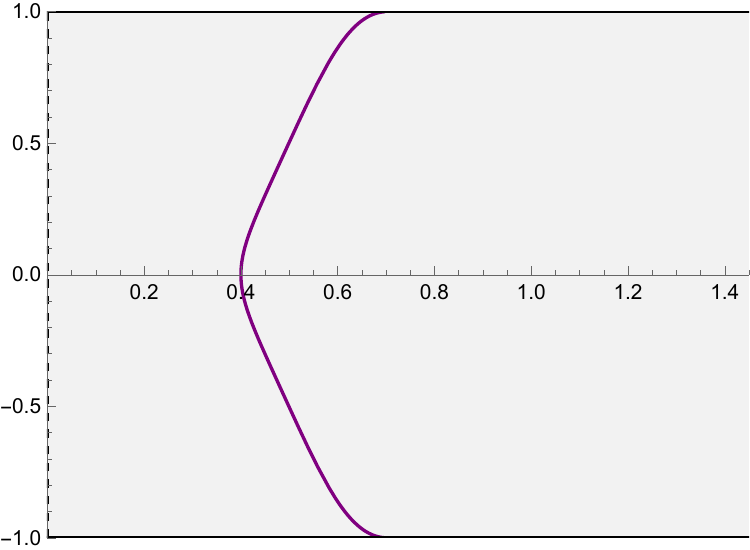}
        \caption{Phase spaces for the choice $c_0=1$ and $\h(y)=y^2+1$, $\Theta_1$ on the left and $\Theta_{-1}$ on the right. The dotted curve corresponds to $\Gamma_1$ and the red curve to the orbit $\gamma_0$. The blue curve is the orbit corresponding to the $\mathcal{H}$-surface with unduloid structure. The purple curves correspond to the $\mathcal{H}$-surfaces with nodoid structure.}
        \label{orbitspositive}
    \end{figure}
    
    Now, we analyse the orbits in $\Theta_1$ and their associated $\mathcal{H}$-surfaces. First, observe that the equilibrium $e_0 = \left(\frac{1}{2\mathfrak{h}(0)},0\right) \in \Theta_1$ corresponds to the cylinder $C_\mathcal{H} = \mathbb{S}^1\left(\frac{1}{2\mathfrak{h}(0)}\right) \times \R$.
    
    The linearized system at $e_0$ associated to the nonlinear system \eqref{mainsystem} for $\varepsilon=1$ is
    \begin{equation}\label{positivelinearsystem}
         \left( \begin{array}{c} u \\ v \end{array}\right)' = \left( \begin{array}{cc}
             0 & 1 \\
             -\frac{4\mathfrak{h}(0)^2}{1+4c_0^2\mathfrak{h}(0)^2} & 0
         \end{array} \right) \left( \begin{array}{c} u \\ v \end{array}\right),
    \end{equation}
    whose orbits are ellipses around the origin. By classical theory of nonlinear autonomous systems, since the eigenvalues of the Jacobian matrix in \eqref{positivelinearsystem} are pure complex numbers $\pm \frac{2\mathfrak{h}(0)}{\sqrt{1+4\mathfrak{h}(0)^2}}i$, we get that the equilibrium $e_0$ has a center structure. This means that all orbits around $e_0$ are closed curves, since $\Theta_1$ is symmetric with respect to the $y=0$ axis. In particular, we deduce that all orbits of $\Theta_1$ stay at a positive distance from $e_0$.
    
    By Proposition \ref{propsurfaceaxis}, there exists an unique orbit $\gamma_0$ in the phase space $\Theta_1$ that passes through $p_1=\left(0,\frac{1}{\sqrt{1+c_0^2\mathfrak{h}(0)^2}}\right)$. By monotonicity properties, $\gamma_0$ lies near $p_1$ inside the monotonicity region $\Lambda_2$ and it cannot stay forever in that region, by a same previous argument using that $\mathcal{H} \geq H_0 > 0$. Moreover, $\gamma_0$ stays at a positive distance from $e_0$ and it does not intersect $\Gamma_1$. Thus, $\gamma_0$ must intersect the $y=0$ axis at some point $(x_0,0)$ with $x_0 > \frac{1}{2\mathfrak{h}(0)}$. 
    
    By symmetry of the phase space, we deduce that $\gamma_0$ can be extended to be an orbit in $\Theta_1$ that joins the limit points $p_1$ and $-p_1$, and that lies in the region $\Lambda_2 \cup \Lambda_3 \cup \{y=0\}$. Taking into account that we can extend this orbit to $x(s)<0$, using that the phase space is symmetric by Remark \ref{obsnegative}, the helicoidal $\mathcal{H}$-surface associated to the orbit $\gamma_0$ is given by the orbit of a curve $\alpha(s) = (x(s),0,z(s))$ under a helicoidal motion, with $x(s)$ periodic and $x(s_0 + kT) = 0$, for some $s_0, T \in \R$ and for all $k \in \mathbb{N}$, such that $z'(s) > 0$, for all $s \in \R$. See Figure \ref{positiveeixo}.
    
    \begin{figure}[h]
        \centering
        $\vcenter{\hbox{\includegraphics[scale=0.7]{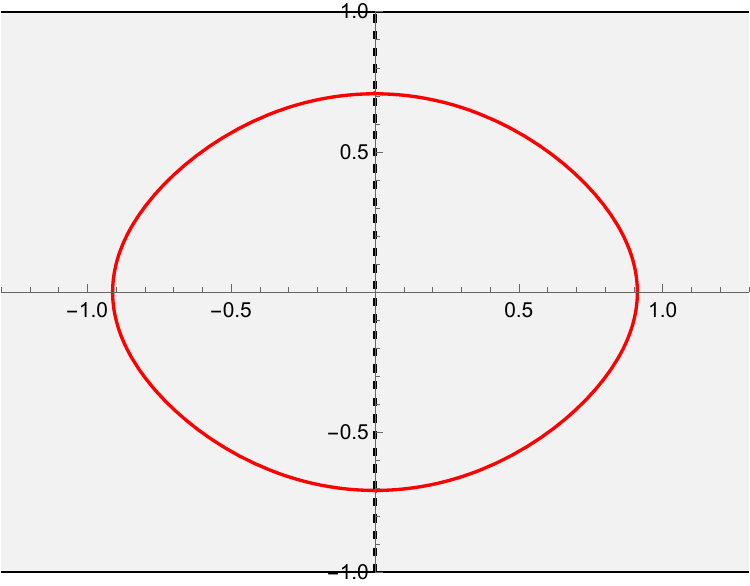}}}
        \hspace{1cm}
        \vcenter{\hbox{\includegraphics[scale=0.5]{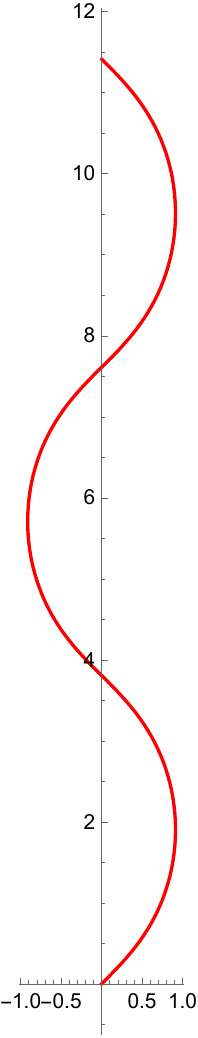}}}
        \hspace{1cm}
        \vcenter{\hbox{\includegraphics[scale=0.61]{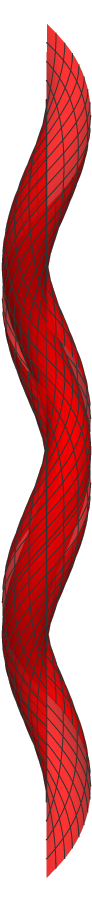}}}$
        \caption{The orbit $\gamma_0$ in $\Theta_1$, its associated profile curve and its resulting helicoidal $\mathcal{H}$-surface.}
        \label{positiveeixo}
    \end{figure}
    
    The orbit $\gamma_0$ divides $\Theta_1$ into two connected components. The first one containing the equilibrium $e_0$, which we will denote by $\mathcal{W}_0$ and the second one, where $x>0$ is unbounded, which we will denote by $\mathcal{W}_\infty$. By uniqueness of a system solution, any orbit of $\Theta_1$ other than $\gamma_0$ lies entirely in $\mathcal{W}_0$ or $\mathcal{W}_\infty$.
    
    We can observe, by symmetry and monotonicity, that any orbit in $\mathcal{W}_\infty$ is a symmetric horizontal graph $x=g(y)$ with $g(1)=g(-1)=x_0$ for some $x_0>0$, and such that $g$ is strictly increasing in $(-1,0)$ and strictly decreasing in $(0,1)$. The resulting $\mathcal{H}$-surface $\Sigma_1$ in $\R^3$ associated to any such orbit is then a surface with boundary parametrized by the curves 
    $$\gamma_1(t)=(x_0\cos t, x_0\sin t, c + c_0 t) \quad \text{and} \quad \gamma_2(t)=(x_0\cos t, x_0\sin t, d + c_0 t), \quad t \in \R,$$
    for some $c<d$. The $z(s)$-coordinate of the profile curve $\alpha(s)$ of $\Sigma_{1}$ is strictly increasing and the unit normal to $\Sigma_{1}$ along $\gamma_1$ and $\gamma_2$ are, respectively,
    $$\eta_1(t) = \frac{(c_0\sin t,-c_0\cos t,x_0)}{\sqrt{c_0^2+x_0^2}} \quad \text{and} \quad \eta_2(t) = \frac{(c_0\sin t,-c_0\cos t,-x_0)}{\sqrt{c_0^2+x_0^2}}.$$
    
    Consequently, by uniqueness of the solution to the Cauchy problem for helicoidal $\mathcal{H}$-surfaces, we can deduce that, given $x_0>0$, the $\mathcal{H}$-surfaces $\Sigma_{-1}$ and $\Sigma_1$ constructed associated to $x_0$ can be smoothly glued together along any of boundary components where their unit normal coincide, to form a larger $\mathcal{H}$-surface. For this, we can assume without loss of generality in the previous construction that $a=c$ or $b=d$, and hence $\Sigma_1$ and $\Sigma_{-1}$ have the same Cauchy data.
    
    If we have simultaneously $a=c$ and $b=d$, the profile curve of the $\mathcal{H}$-surface obtained is diffeomorphic to the circle $\mathbb{S}^1$. Thus, the resulting $\mathcal{H}$-surface $\Sigma$ will be a tube diffeomorphic to the cylinder $\S^1\times \R$. Note that $\Sigma$ could be properly embedded and it is important to observe that, if $\mathcal{H} \equiv H_0 > 0$, this is impossible by Theorem 2.10 of \cite{SOLOMON}.
    
    If we have $a=c$ and $b \neq d$, or $a \neq c$ and $ b = d$, by iterating the previous process we obtain a non-embedded helicoidal $\mathcal{H}$-surface, with a nodoid structure, invariant by a vertical translation. See Figure \ref{positiveneg}.
    
    \begin{figure}
        \centering
        $\vcenter{\hbox{\includegraphics[scale=0.6]{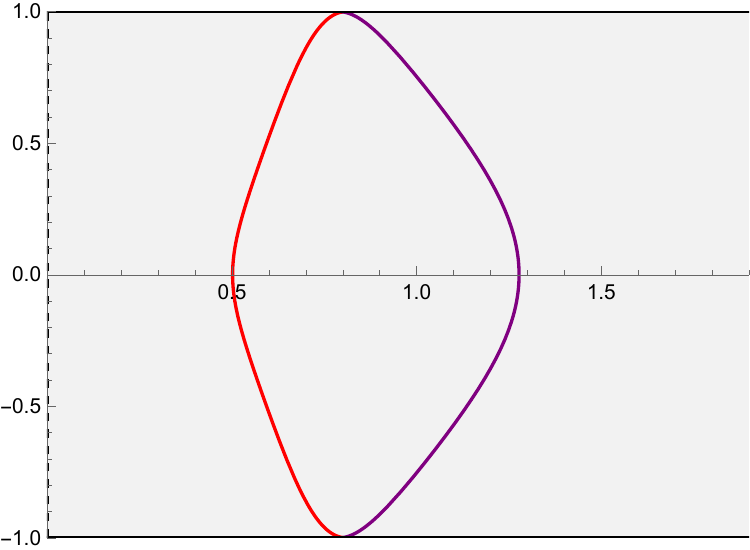}}}
        \hspace{0.8cm}
        \vcenter{\hbox{\includegraphics[scale=0.4]{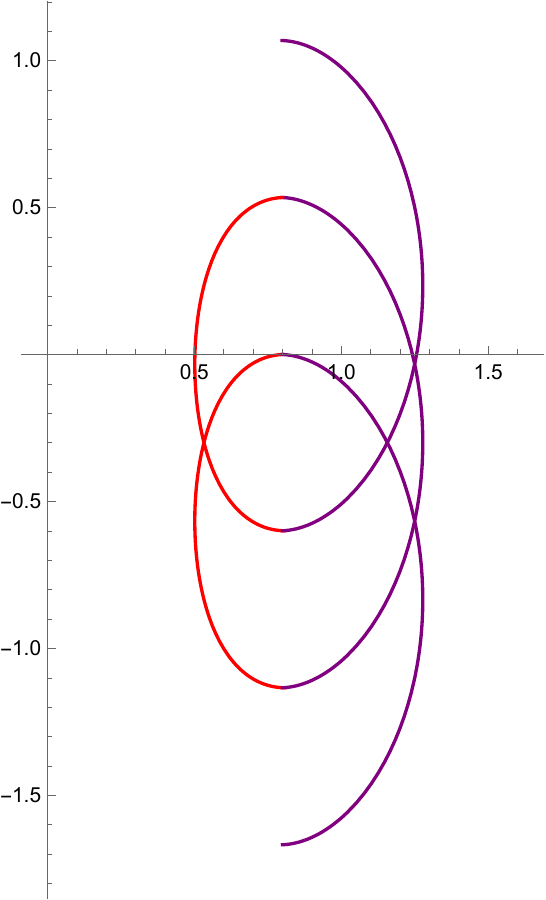}}}
        \hspace{0.8cm}
        \vcenter{\hbox{\includegraphics[scale=0.45]{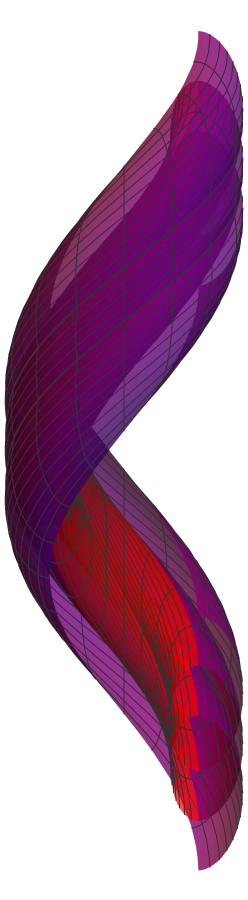}}}$
        \caption{On the left, the red curve corresponds to an orbit in the $\Theta_{-1}$ phase space and the purple curve to an orbit in $\Theta_1$ that lies in $\mathcal{W_\infty}$. On the middle, their associated profile curves. On the right, the red surfaces correspond to the resulting helicoidal $\mathcal{H}$-surfaces $\Sigma_{-1}$ and the purple surfaces to $\Sigma_1$, composing a helicoidal $\mathcal{H}$-surface with nodoid structure.}
        \label{positiveneg}
    \end{figure}
    
    Now, we consider an orbit $\gamma$ of $\Theta_1$ that is contained in the region $\mathcal{W}_0$. Since $\gamma$ stays at a positive distance from $e_0$ and cannot approach a point of the form $(0,y)$, with $y \neq \pm \frac{1}{\sqrt{1+c_0^2\mathfrak{h}(0)^2}}$, we deduce, taking into account the monotonicity properties and the symmetry with respect to the $y=0$ axis, that $\gamma$ is a closed curve containing $e_0$ in its inner region. That means that the profile curve $\alpha(s) = (x(s),0,z(s))$ of the helicoidal $\mathcal{H}$-surface associated to any such orbit has an unduloid structure, that is, satisfies that $s$ is defined for all real values, such that $z'(s)>0$ for all $s\in \R$ and $x(s)$ is periodic. See Figure \ref{positivepos}.
    \begin{figure}
        \centering
        $\vcenter{\hbox{\includegraphics[scale=0.72]{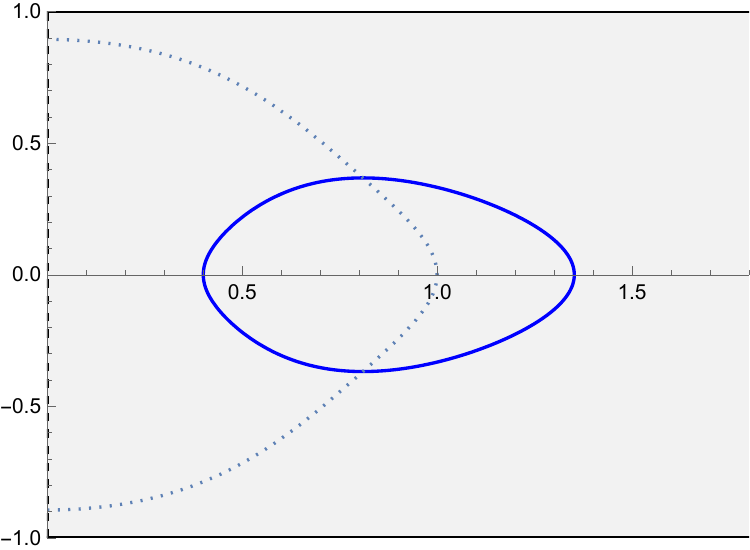}}}
        \hspace{1.2cm}
        \vcenter{\hbox{\includegraphics[scale=0.13]{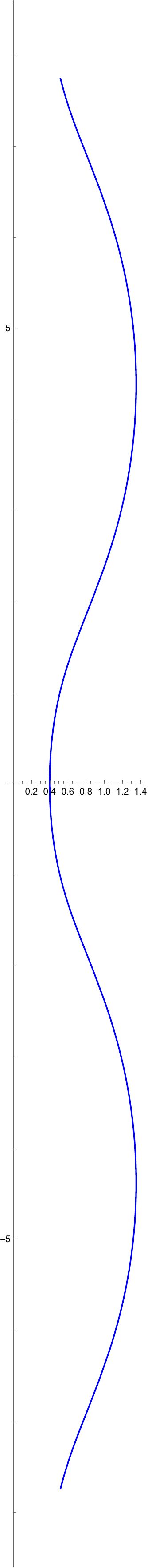}}}
        \hspace{1.2cm}
        \vcenter{\hbox{\includegraphics[scale=0.55]{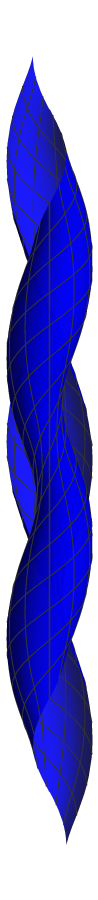}}}$
        \caption{An orbit that lies in $\mathcal{W}_0$ in the $\Theta_1$ phase space, its associated profile curve and its resulting helicoidal $\mathcal{H}$-surface, which has an unduloid structure.}
        \label{positivepos}
    \end{figure}
\end{proof}

\begin{remark}
    It is important to observe that the fact that $\h$ is increasing in $[0,1]$ is only used to know that $\Gamma_1$ is regular and connected.  If it were established that this behavior of $\Gamma_1$ persists even when the function $\h$ does not exhibit monotonicity, the above theorem would remain valid.

    For instance, the Figure \ref{nullcline_exe} illustrates the behavior of $\Gamma_1$ for the specific choice $\h(t) = \cos(40t) + 2$, an even and positive $C^1$ function that is not increasing in $[0,1]$. Observe that $\Gamma_1$ is not connected in this case.
    \begin{figure}
        \centering
        \includegraphics[scale=0.5]{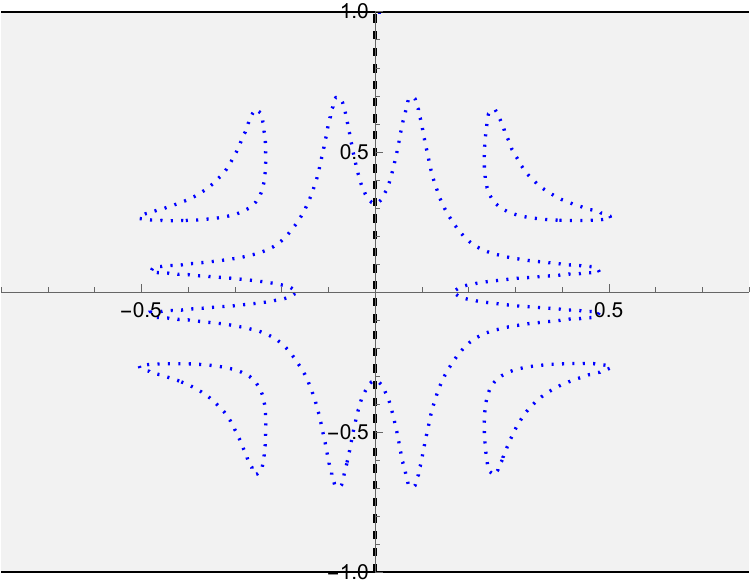}
        \caption{The curve $\Gamma_1$ for the choice $\h(t) = \cos(40t) + 2$.} \label{nullcline_exe}
    \end{figure}
    \qed
\end{remark}


\section{Another examples of helicoidal \texorpdfstring{$\mathcal{H}$}{H}-surfaces}\label{sectionexamples}

In this section, we will exhibit some examples of helicoidal $\mathcal{H}$-surfaces which there exists some $t_0 \in (-1,1)$ ($t_0 = \pm 1$ is not included by Remark \ref{remarkang1}) such that $\h(t_0) = 0$ and $\h(t) \neq 0$, for all $t \neq t_0$, where $\h \in C^1([-1,1])$ is given by \eqref{relationHh} in terms of $\mathcal{H}$.

By \eqref{anglefunction}, the angle function of $\Sigma$ is constant along the curves
$$\beta_0 := \left\{ (x,y) \in \Theta_\varepsilon : \frac{xy}{\sqrt{x^2+c_0^2y^2}}=t_0\right\}$$
in each $\Theta_\varepsilon$ phase space. It is important to observe that, in the $x=0$ axis, the angle function vanishes, in particular, at the points $\pm p_\varepsilon$, when they exist. Since $x>0$ in $\Theta_\varepsilon$, note that we can describe $\beta_0$ as a graph $y=f(x)$, where
$$f(x) = \frac{t_0x}{\sqrt{x^2-c_0^2t_0^2}}.$$

By letting $x\to \infty$ in $f(x)$, we get that $\beta_0$ converges to the straight line $y=t_0$. Indeed,
$$\lim_{x\to\infty}f(x) = \lim_{x\to\infty}\frac{t_0x}{\sqrt{x^2-c_0^2t_0^2}} = \lim_{x\to\infty}\frac{t_0}{\sqrt{1-\frac{t_0^2c_0^2}{x^2}}}=t_0.$$

Along $\beta_0$, letting $y\to\pm 1$, we would have $\frac{x}{\sqrt{c_0^2+x(s)^2}}\to \pm t_0$, that is, $x\to\frac{\pm |c_0|t_0}{\sqrt{1-t_0^2}}$. Set
    $$p_0^+:=\left( \frac{|c_0|t_0}{\sqrt{1-t_0^2}}, 1 \right), \text{ if } t_0\geq 0, \text{ and  } p_0^-:=\left( \frac{-|c_0|t_0}{\sqrt{1-t_0^2}}, -1 \right), \text{ if } t_0 \leq 0. \text{ (See Fig. \ref{beta0fig}).}$$
    
\begin{figure}[h]
    \centering
    $\vcenter{\hbox{\includegraphics[scale=2.1]{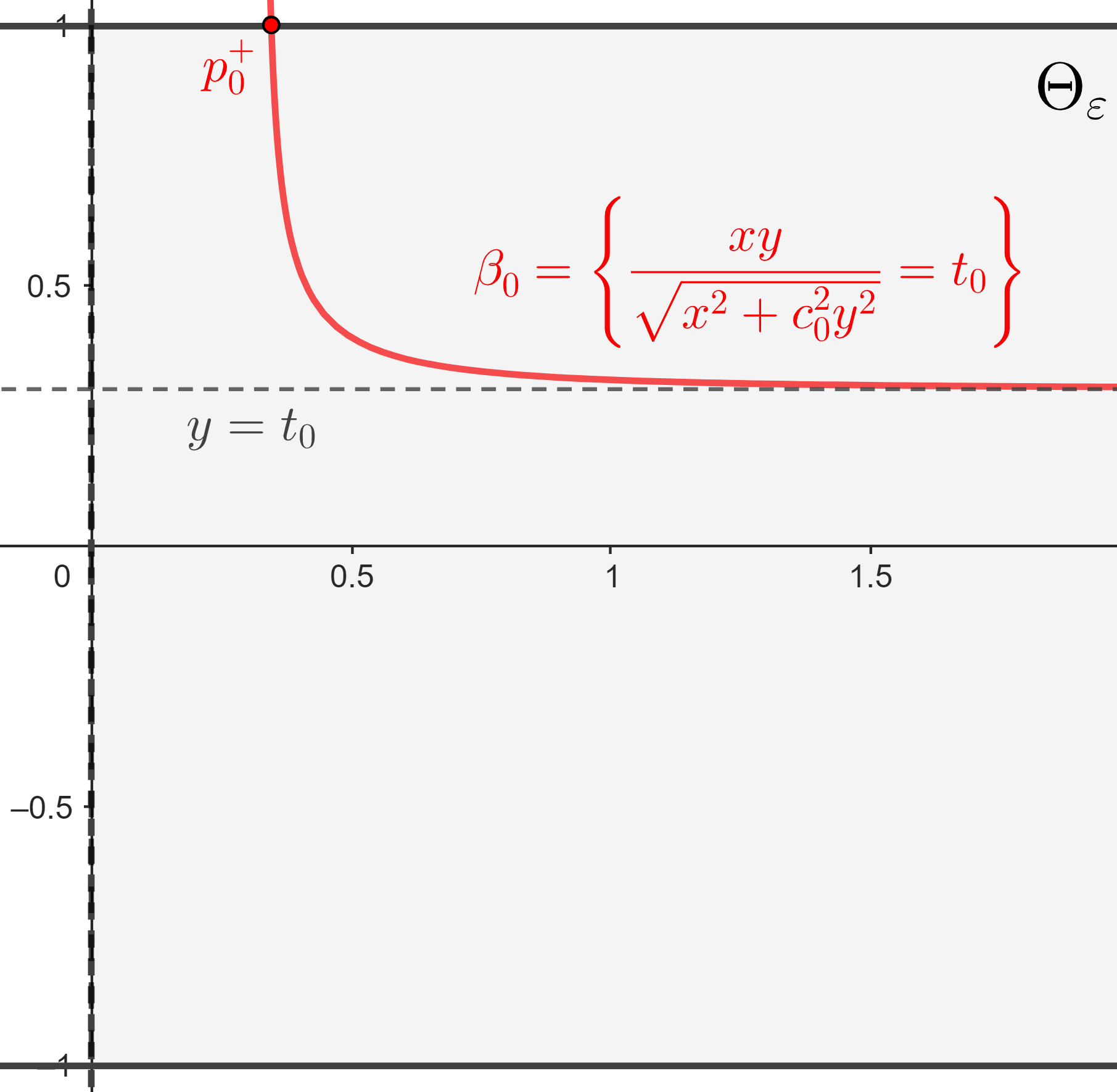}}}
    \hspace{0.7cm}
    \vcenter{\hbox{\includegraphics[scale=2.1]{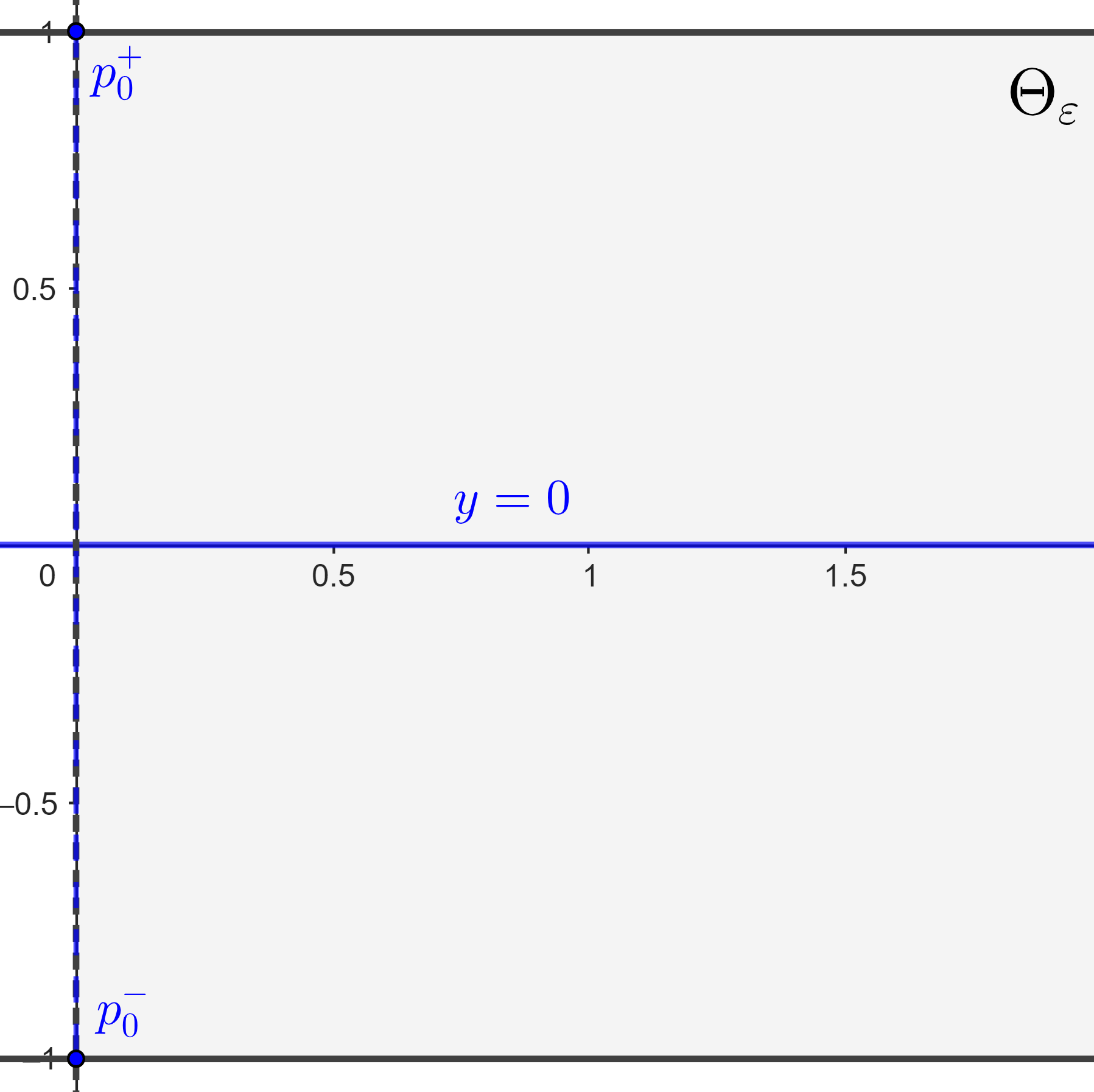}}}
    \hspace{0.7cm}
    \vcenter{\hbox{\includegraphics[scale=2.1]{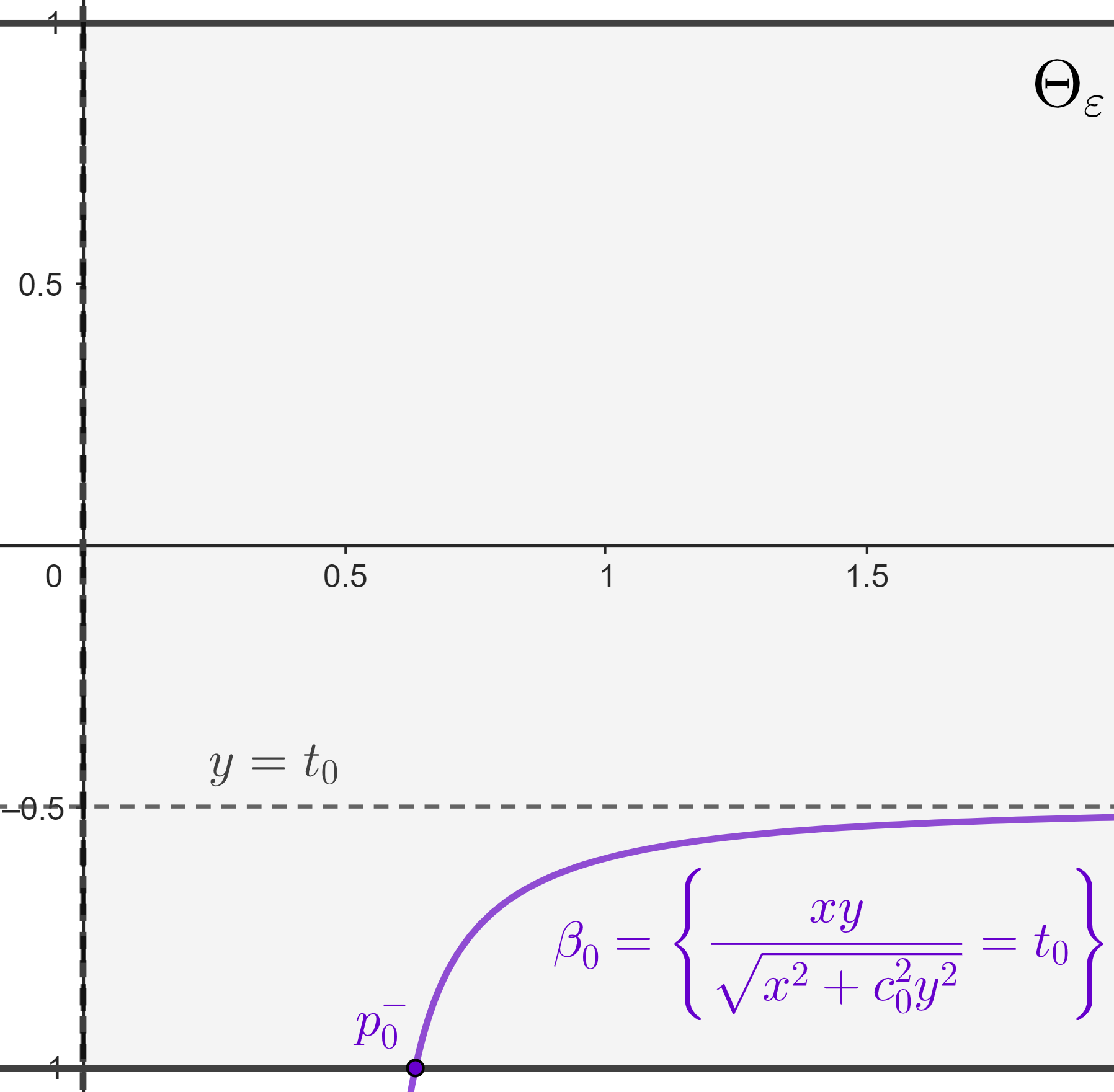}}}$
    \caption{A sketch for the curve $\beta_0$ for each $t_0 > 0$, $t_0 = 0$ and $t_0<0$.}
    \label{beta0fig}
\end{figure}

It is important to observe that, if $\Gamma_\varepsilon$ exists, it must not intersect the curve $\beta_0$. Indeed, if $(x,y) \in \beta_0$, we would get that $\h \left(\frac{xy}{\sqrt{x^2+c_0^2y^2}}\right)=0$ in \eqref{formgamma}, which is impossible. Moreover, letting $y \to \pm 1$ in \eqref{formgamma}, observe that $\Gamma_\varepsilon$ converges to $p_0^+$ or $p_0^-$.

\subsection{Case \texorpdfstring{$\h(0)=0$}{h(0)=0}}

Suppose $t_0=0$. In this case, we get that $\pm p_\varepsilon=(0,\pm 1)$ and there is not a equilibrium in both $\Theta_\varepsilon$ phase spaces. Moreover, if $\Gamma_\varepsilon$ exists, it must intersect $(0,\pm 1)$. Letting $y \to 0$ in \eqref{formgamma}, we obtain that $\Gamma_\varepsilon$ converges at infinity to the $y=0$ axis.

First, consider $\h \in C^1([-1,1])$ such that $\h(0) = 0$ and $\h(t) > 0$, for all $t \neq 0$. This means that $\Gamma_{-1}$ does not exist and $\Gamma_1$ have two connected components, each one contained in the regions $\{y>0\}$ and $\{y<0\}$. Consider the example $\h(t) = t^2$.

The curve $\Gamma_1$ in $\Theta_1$ together with $y=0$ divides the phase space $\Theta_1$ into four monotonicity regions $\Lambda_1, \dots, \Lambda_4$, where $\Lambda_1$ and $\Lambda_2$ are contained in the $\{y>0\}$ region and $\Lambda_3$ and $\Lambda_4$ are contained in $\{y<0\}$, such that $\Lambda_1$ and $\Lambda_4$  meet at the $y=0$ axis. The curve $\Gamma_{-1}$ in $\Theta_{-1}$ does not exist, so $\Theta_{-1}$ has only two monotonicity regions $\Lambda_+ := \Theta_{-1} \cap \{ y>0 \}$ and $\Lambda_- := \Theta_{-1} \cap \{ y<0 \}$. See Figure \ref{Phasespace_parabola}.

\begin{figure}[h]
    \centering
    $\vcenter{\hbox{\includegraphics[scale=0.55]{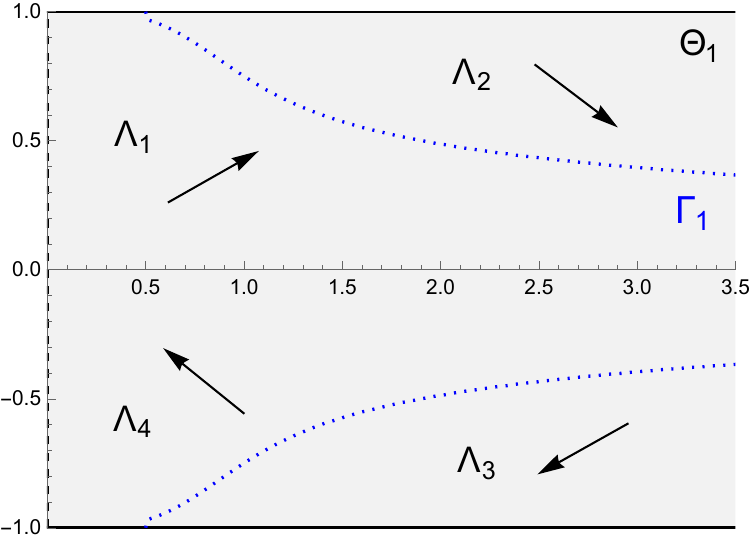}}}
    \hspace{0.7cm}
    \vcenter{\hbox{\includegraphics[scale=0.55]{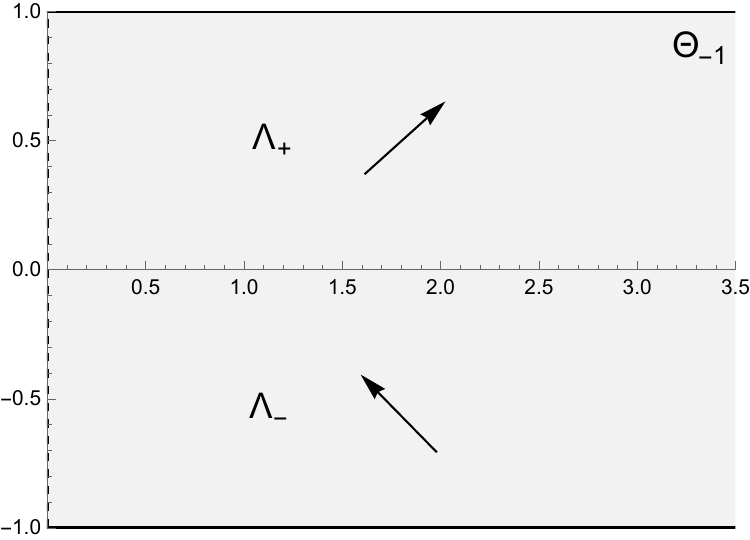}}}$
    \caption{The phase spaces $\Theta_1$ and $\Theta_{-1}$ for the choice $\h(t)=t^2$.}
    \label{Phasespace_parabola}
\end{figure}

By Proposition \ref{propsurfaceaxis}, let $\gamma_0(s)=(x_0(s),y_0(s))$ be the orbit that corresponds to the profile curve of the helicoidal $\mathcal{H}$-surface that meets its rotation axis. Up to a change of orientation, assume that $\gamma_0$ meets the $x=0$ axis at the point $p_1 = (0, 1)$. By Remark \ref{obsnegative}, we can analyse its behavior for $x_0(s)<0$ in the regions $\Lambda_j^*$, $j = 1,\dots, 4$, where $\Lambda_j^*$ is the reflection of $\Lambda_j$ with respect to the origin.

For $x(s)>0$, we obtain, by monotonicity properties, that $\gamma_0$ lies entirely in the region $\Lambda_2$ and converges to the axis $y=0$ as $s \to \infty$. For $x(s)<0$, we get the same behavior, noting that $\gamma_0$ lies in the region $\Lambda_3^*$. See Figure \ref{parabolaeixo}.

\begin{figure}
    \centering
    $\vcenter{\hbox{\includegraphics[scale=0.57]{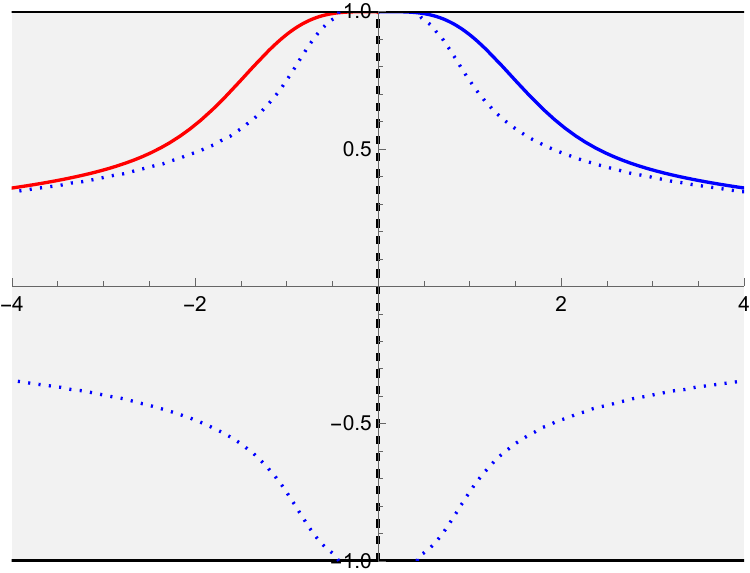}}}
    \hspace{0.8cm}
    \vcenter{\hbox{\includegraphics[scale=0.6]{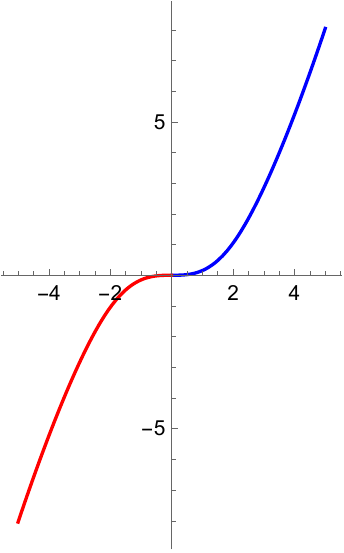}}}
    \hspace{0.8cm}
    \vcenter{\hbox{\includegraphics[scale=0.44]{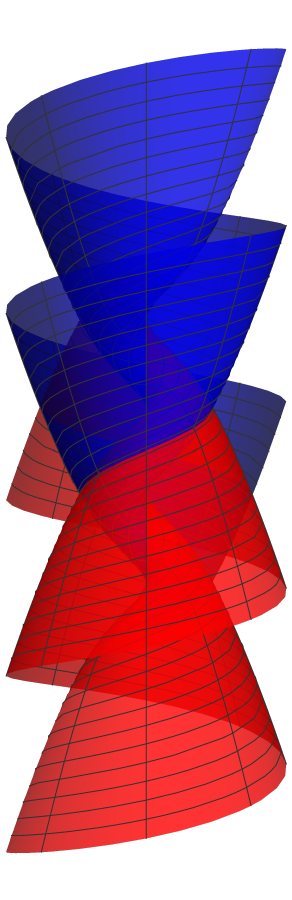}}}$
    \caption{The orbit $\gamma_0(s)$ in $\Theta_1$, where the blue component correspond to $x_0(s)>0$ and the red component to $x_0(s)<0$, its associated profile curve and its resulting helicoidal $\mathcal{H}$-surface, for the choice $\h(t) = t^2.$}
    \label{parabolaeixo}
\end{figure}

Now, for each $\varepsilon=\pm 1$, let $\gamma_\varepsilon(s)=(x_\varepsilon(s),y_\varepsilon(s))$ be the orbit that passes through $(r_0,0)$, for some $r_0>0$, and belongs to the phase space $\Theta_\varepsilon$ around that point. By monotonicity properties, note that, for $y_1 > 0$, the orbit $\gamma_1$ must intersect $\Gamma_1$, moving for the $\Lambda_2$ region and, then, converges at infinity to the $y=0$ axis, getting the same behavior for $y_1<0$. See Figure \ref{parabolapos}.

\begin{figure}
    \centering
    $\vcenter{\hbox{\includegraphics[scale=0.65]{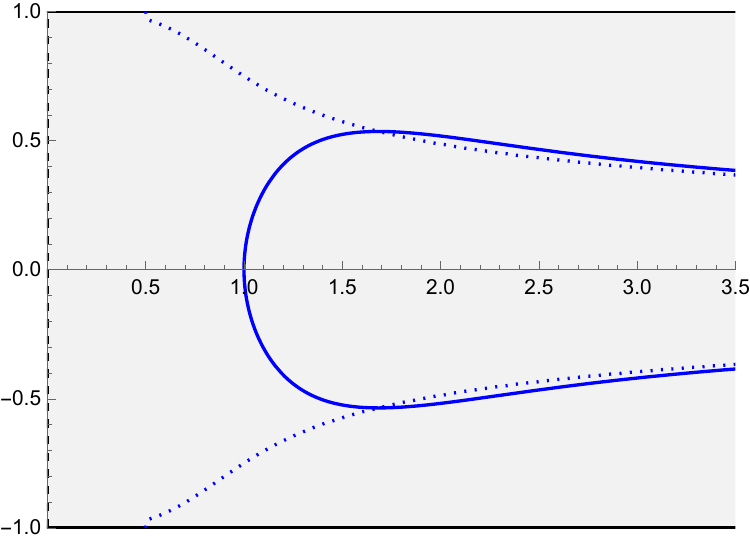}}}
    \hspace{0.8cm}
    \vcenter{\hbox{\includegraphics[scale=0.32]{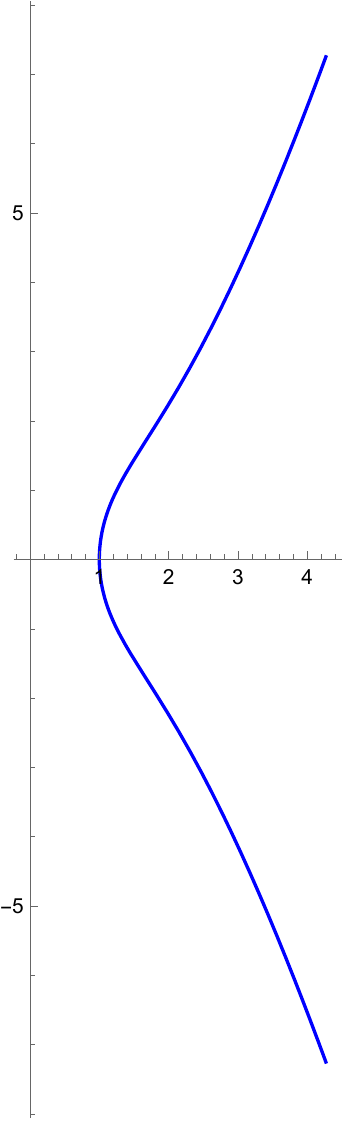}}}
    \hspace{0.8cm}
    \vcenter{\hbox{\includegraphics[scale=0.45]{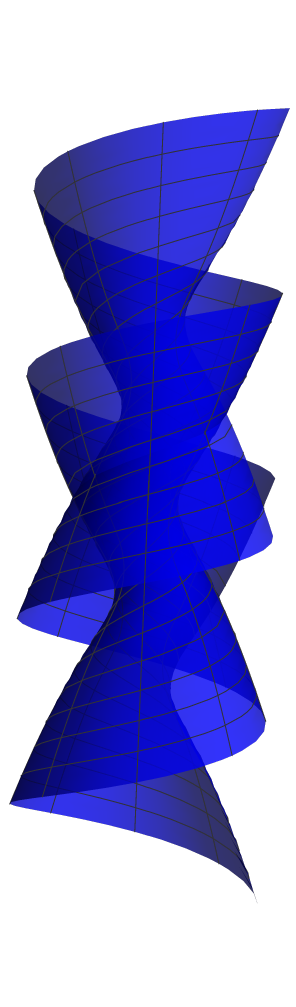}}}$
    \caption{The orbit $\gamma_1$ in $\Theta_1$, for the choices $\h(t) = t^2$ and $r_0=1$, its associated profile curve and its resulting helicoidal $\mathcal{H}$-surface.}
    \label{parabolapos}
\end{figure}

Observe that the orbit $\gamma_{-1}$ in $\Theta_{-1}$ is given by a horizontal graph $x=g(y)$, where $g \in C^1((a,b))$, for some $-1\leq a<0<b \leq 1$, such that $g(0)=r_0$ and $g$ restricted to $[0,b]$ is strictly increasing (resp. $g$ restricted to $[a,0]$ is strictly decreasing). Suppose that $g(y) \to \infty$ as $y \to \{a,b\}$. Consider the parametrization $\Psi(s,\theta)$ of the resulting helicoidal $\mathcal{H}$-surface $\Sigma$, as in \eqref{parametrization}. If we restrict $\theta \in (-\frac{\pi}{2},\frac{\pi}{2})$ and exclude the helix $\gamma(t)= (r_0\cos t, r_0\sin t, z_0 + c_0 t) $, the surface $\Sigma$ would be a bi-graph over the set 
$$\Omega = \R^2 \smallsetminus B_{r_0+\delta}(0) \cap \{ x>0 \},$$
for some $\delta>0$, such that $\mathcal{H}\geq H_0 >0$ in $\Omega$. This is impossible, by previous discussions using the maximum principle.

Therefore, $a = -1$, $b = 1$ and any orbit in $\Theta_{-1}$ has two limit endpoints of the form $(x_1, 1)$ and $(x_2,-1)$, for some $x_1,x_2 > 0$. In these points, the coordinate $z(s)$ attains, respectively, local maximum and local minimum values, so we can analyse orbits $\gamma_{-1}^{(1)}$ and $\gamma_{-1}^{(2)}$ in $\Theta_{1}$ that have, respectively, $(x_1, 1)$ and $(x_2,-1)$ as limit endpoints and, then, smoothly glue together the resulting $\mathcal{H}$-surfaces. By monotonicity properties of $\Theta_1$, the orbit $\gamma_{-1}^{(1)}$ stays in the region $\Lambda_2$ and converges at infinity to the $y=0$ axis as well as $\gamma_{-1}^{(2)}$ stays in the region $\Lambda_3$ and  converges at infinity to the same axis. See Figure \ref{parabolaneg}.

\begin{figure}
    \centering
    $\vcenter{\hbox{\includegraphics[scale=0.6]{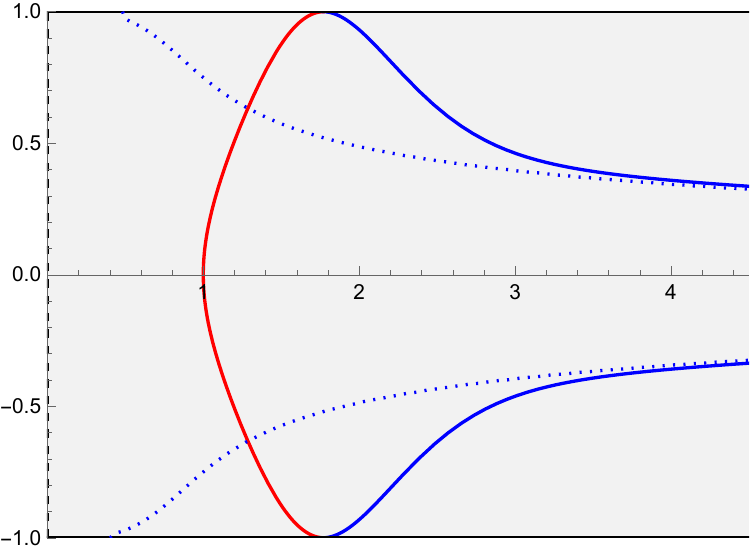}}}
    \hspace{0.8cm}
    \vcenter{\hbox{\includegraphics[scale=0.52]{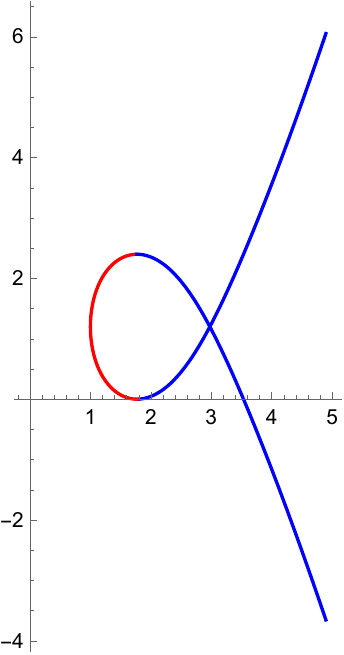}}}
    \hspace{0.8cm}
    \vcenter{\hbox{\includegraphics[scale=0.43]{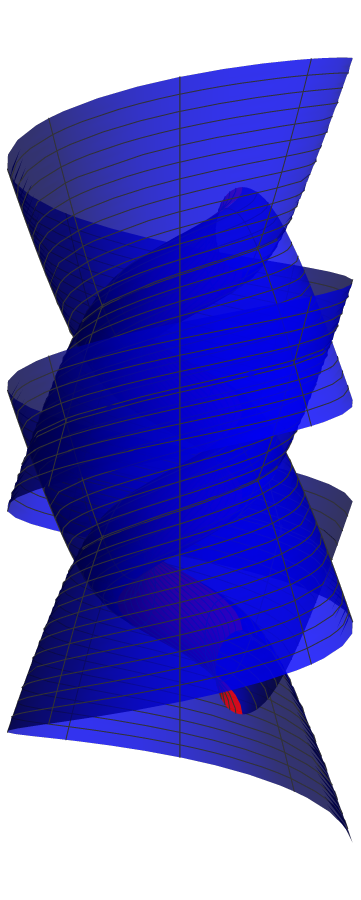}}}$
    \caption{On the left, the red curve corresponds to the orbit $\gamma_{-1}$ in $\Theta_{-1}$, for the choices $\h(t) = t^2$ and $r_0=1$. The blue curves correspond to the orbits $\gamma_{-1}^{(1)}$ and $\gamma_{-1}^{(2)}$ in $\Theta_1$. On the middle, their associated profile curve and, on the right, their resulting helicoidal $\mathcal{H}$-surface.}
    \label{parabolaneg}
\end{figure}

Now, consider $\h \in C^1([-1,1])$ such that $\h(0) = 0$, $\h(t) > 0$, for all $t > 0$ and $\h(t) < 0$, for all $t < 0$. This means that $\Gamma_1$ is contained in the region $\{y>0\}$ and $\Gamma_{-1}$ is contained in the region $\{y<0\}$. Consider the example $\h(t) = t$.

The curve $\Gamma_1$ in $\Theta_1$ together with $y=0$ divides the phase space $\Theta_1$ into three monotonicity regions $\Lambda_1^+$, $\Lambda_2^+$ and  $\Lambda_3^+$, where $\Lambda_1^+ := \Theta_1 \cap \{y<0\}$ and $\Lambda_2^+$ and  $\Lambda_3^+$ are contained in the $\{y>0\}$ region, such that $\Lambda_2^+$ and $\Lambda_1^+$ meet at the $y=0$ axis. In the phase space $\Theta_{-1}$, it happens an analogous partition, the curve $\Gamma_{-1}$ together with $y=0$ divides the phase space $\Theta_{-1}$ into three monotonicity regions $\Lambda_1^-$, $\Lambda_2^-$ and  $\Lambda_3^-$, where $\Lambda_3^- := \Theta_{-1} \cap \{y>0\}$ and $\Lambda_1^-$ and  $\Lambda_2^-$ are contained in the $\{y<0\}$ region, such that $\Lambda_2^-$ and $\Lambda_3^-$ meet at the $y=0$ axis. See Figure \ref{Phasespace_identidade}.

\begin{figure}
    \centering
    $\vcenter{\hbox{\includegraphics[scale=0.55]{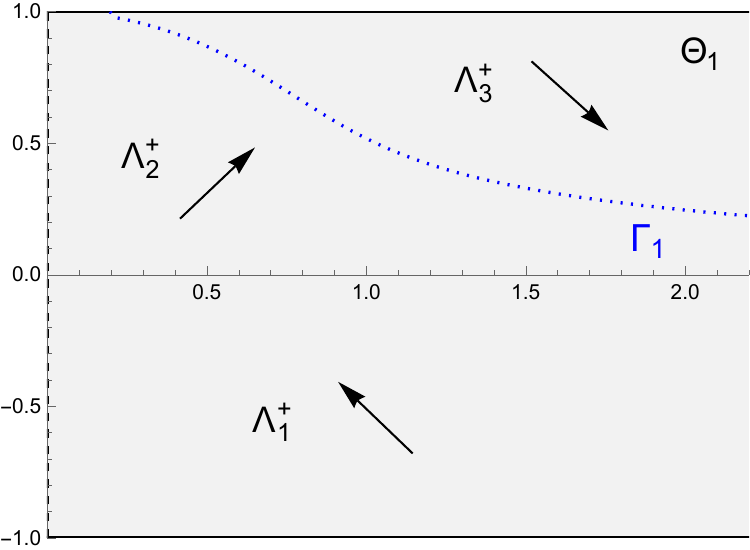}}}
    \hspace{0.7cm}
    \vcenter{\hbox{\includegraphics[scale=0.55]{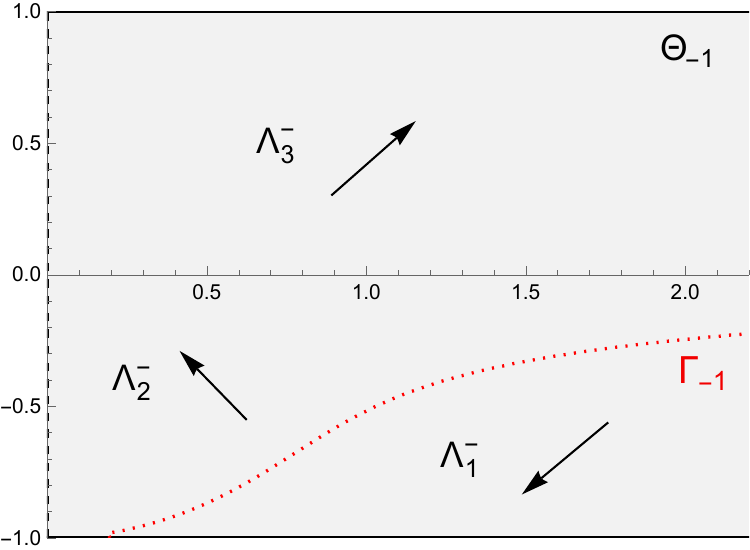}}}$
    \caption{The phase spaces $\Theta_1$ and $\Theta_{-1}$ for the choice $\h(t)=t$.}
    \label{Phasespace_identidade}
\end{figure}

By Proposition \ref{propsurfaceaxis}, let $\gamma_0(s)=(x_0(s),y_0(s))$ be the orbit that corresponds to the profile curve of the helicoidal $\mathcal{H}$-surface that meets its rotation axis. Up to a change of orientation, assume that $\gamma_0$ meets the $x=0$ axis at the point $p_1 = (0, 1)$. We can analyse the orbit behavior for $x_0(s)<0$ in the reflected monotonicity regions with respect to the origin, by previous discussion. 

Observe that, by monotonicity properties, the orbit $\gamma_0$ cannot stay in the phase space $\Theta_1$ for $x_0(s)<0$. This means that the coordinate $z(s)$ of its associated profile curve attains a local minimum in the rotation axis. So, we can analyse the orbit $\gamma_0^-$ in $\Theta_{-1}$ that has $(0,1)$ as an endpoint and smoothly glue together the resulting $\mathcal{H}$-surfaces. Note that $\gamma_0$ lies entirely in the region $\Lambda_3^+$ and converges at infinity to the $y=0$ axis as well as $\gamma_0^-$ lies entirely in the reflected region of $\Lambda_1^-$ with respect to the origin and converges to the $y=0$ axis as $s \to -\infty$. See Figure \ref{identidadeeixo}.

\begin{figure}
    \centering
    $\vcenter{\hbox{\includegraphics[scale=0.5]{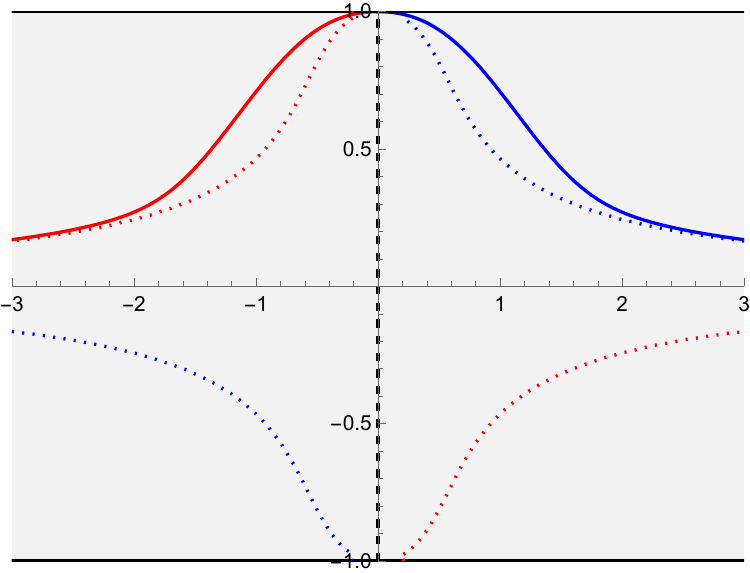}}}
    \hspace{0.8cm}
    \vcenter{\hbox{\includegraphics[scale=0.6]{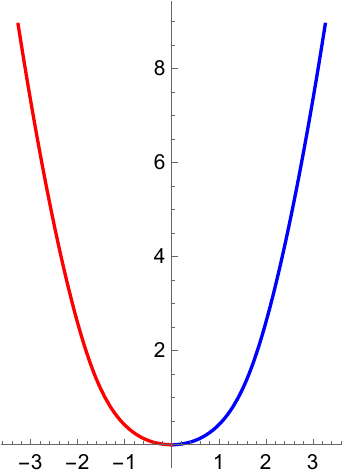}}}
    \hspace{0.8cm}
    \vcenter{\hbox{\includegraphics[scale=0.4]{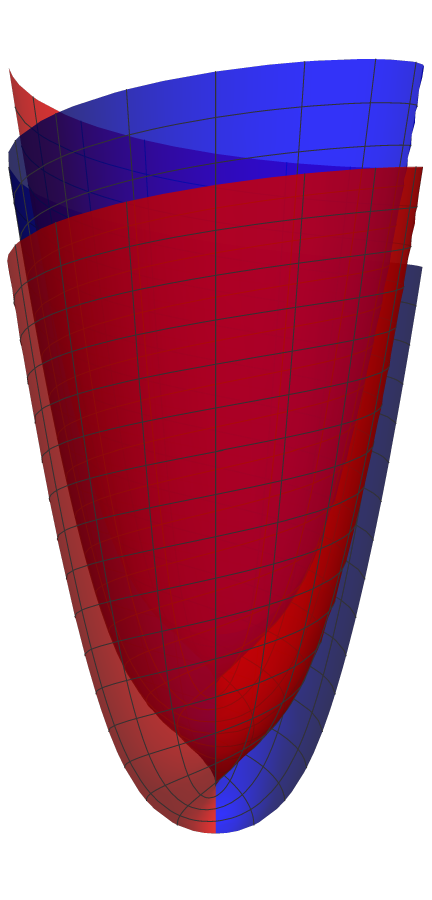}}}$
    \caption{On the left, the blue curve corresponds to the orbit $\gamma_0$ in $\Theta_1$ and the red curve to the orbit $\gamma_0^-$ in $\Theta_{-1}$, for the choice $\h(t)=t$. On the middle, their associated profile curve and, on the right, their resulting helicoidal $\mathcal{H}$-surface.}
    \label{identidadeeixo}
\end{figure}

Now, for each $\varepsilon=\pm 1$, let $\gamma_\varepsilon(s)=(x_\varepsilon(s),y_\varepsilon(s))$ be the orbit that passes through $(r_0,0)$, for some $r_0>0$, and belongs to the phase space $\Theta_\varepsilon$ around this point. By monotonicity properties, note that, for $y_1 > 0$, the orbit $\gamma_1$ must intersect $\Gamma_1$, moving to the region $\Lambda_3^+$ and, then, converges at infinity to the $y=0$ axis. For $y_1 < 0$, the orbit $\gamma_1$ has a limit endpoint of the form $(x_0,-1)$, for some $x_0 > r_0$. Therefore, we analyse an orbit $\gamma_1^-$ in $\Theta_{-1}$ that has $(x_1,-1)$ as a limit endpoint and smoothly glue the resulting $\mathcal{H}$-surfaces. By monotonicity properties, $\gamma_1^-$ stays in the region $\Lambda_1^-$ and converges to the $y=0$ axis as $s\to -\infty$. See Figure \ref{identidadeneg}. The behavior of $\gamma_{-1}$ is pretty similar to $\gamma_{1}$.

\begin{figure}
    \centering
    $\vcenter{\hbox{\includegraphics[scale=0.57]{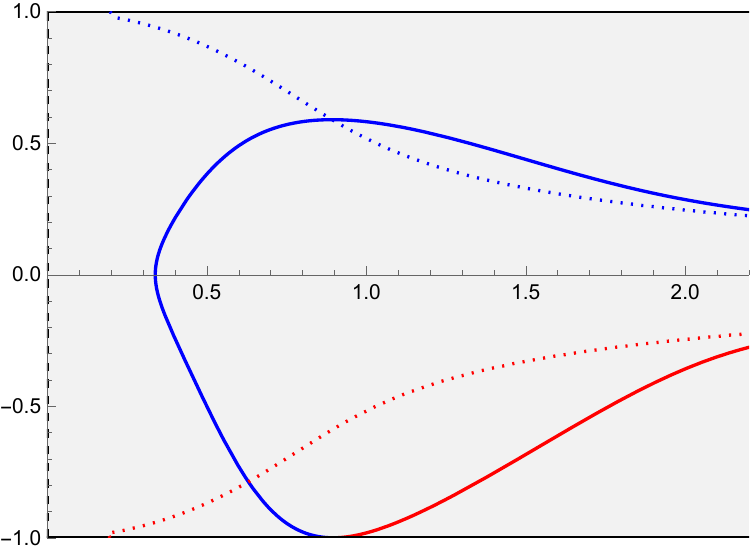}}}
    \hspace{0.8cm}
    \vcenter{\hbox{\includegraphics[scale=0.5]{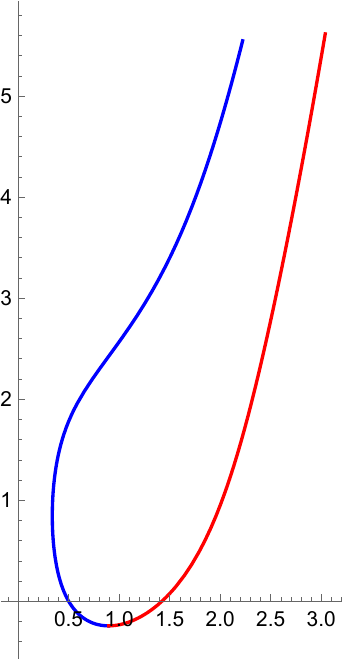}}}
    \hspace{0.6cm}
    \vcenter{\hbox{\includegraphics[scale=0.37]{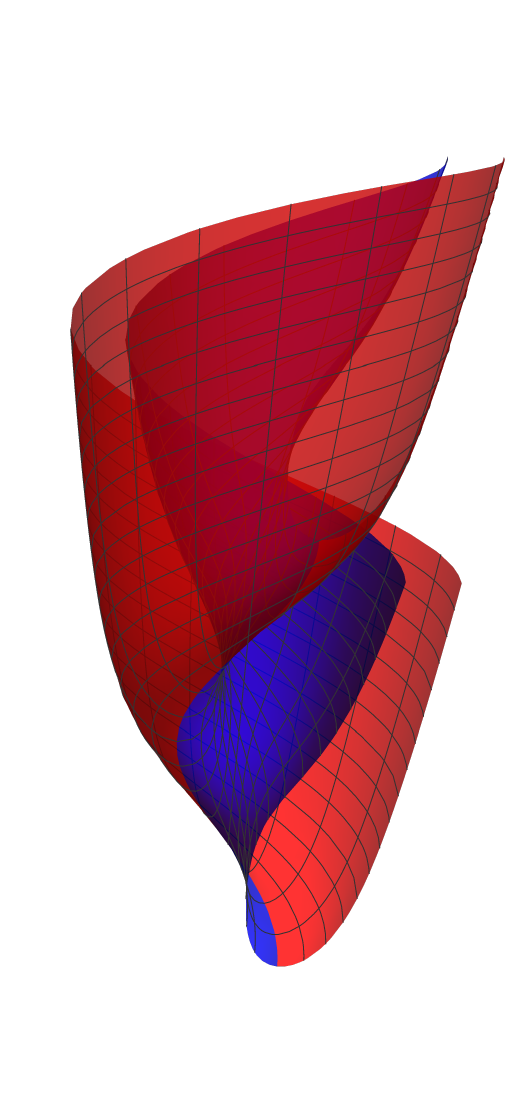}}}$
    \caption{On the left, the blue curve corresponds to the orbit $\gamma_1$ in $\Theta_1$ and the red curve to the orbit $\gamma_1^-$ in $\Theta_{-1}$, for the choice $\h(t)=t$. On the middle, their associated profile curve and, on the right, their resulting helicoidal $\mathcal{H}$-surface.}
    \label{identidadeneg}
\end{figure}

\subsection{Case \texorpdfstring{$\h(t_0)=0$}{h(t0)=0}, for some \texorpdfstring{$t_0 \neq 0$}{t0~=0}}

In this case, suppose that there exists one $t_0 \in (0,1)$ such that $\h(t_0) = 0$ and $\h(t)\neq 0$, for all $t \neq t_0$. Thus, $\h$ vanishes in the curve $\beta_0$ in both phase spaces. First, we will show a example of $\h$ such that $\h(t) > 0$, for all $t \neq t_0$. Finally, we will show a example of $\h$ such that $\h(t) > 0$, for all $t > t_0$, and $\h(t) < 0$, for all $t < t_0$. The biggest problem to extend the orbits behavior of these examples to the general case is not knowing exactly how $\Gamma_\varepsilon$ looks like.  

First, consider the example $\h(t) = (t-0.6)^2$, that is, $t_0 = 0.6$. Thus,
$$p_0^+ = \left( \frac{|c_0|0.6}{\sqrt{1-0.6^2}},1\right) = \left(\frac{3|c_0|}{4},1\right) \quad \text{and} \quad p_1 = \left(0, \frac{1}{\sqrt{1+0.36^2c_0^2}} \right).$$

Since $\h\geq 0$, the curve $\Gamma_{-1}$ does not exist. Therefore, $\Theta_{-1}$ has two monotonicity regions: $\Lambda^+ := \Theta_{-1} \cap \{y>0\}$ and $\Lambda^- := \Theta_{-1} \cap \{y<0\}$. The curve $\Gamma_1$ has three connected components, the first one connecting $p_1$ and $p_0^+$, the second one above the curve $\beta_0$ and the last one containing the equilibrium $e_0 \in \Theta_1$. Then, the curve $\Gamma_1$, together with the $y=0$ axis, divides the phase space $\Theta_1$ into six monotonicity regions $\Lambda_1^+,\dots,\Lambda_4^+,\Lambda_1^-,\Lambda_2^-$, four of them above the $y=0$ axis and two of them below $y=0$. See Figure \ref{Phasespace_pdesloc}.

\begin{figure}[h]
    \centering
    $\vcenter{\hbox{\includegraphics[scale=0.55]{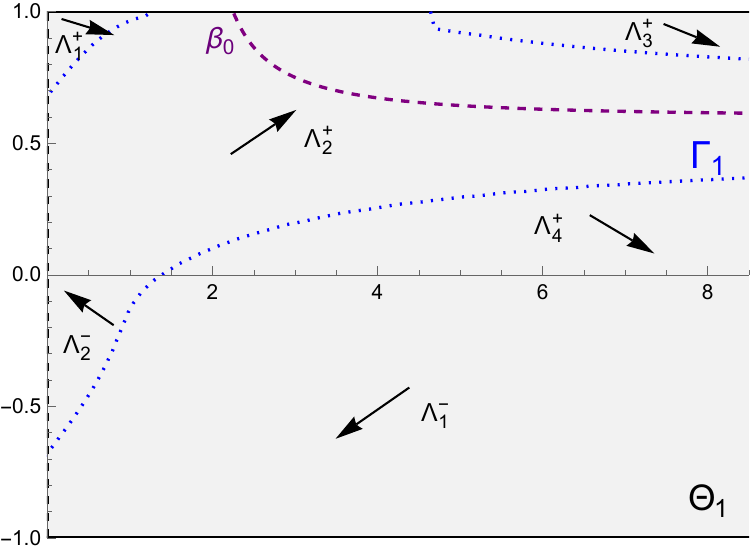}}}
    \hspace{0.7cm}
    \vcenter{\hbox{\includegraphics[scale=0.55]{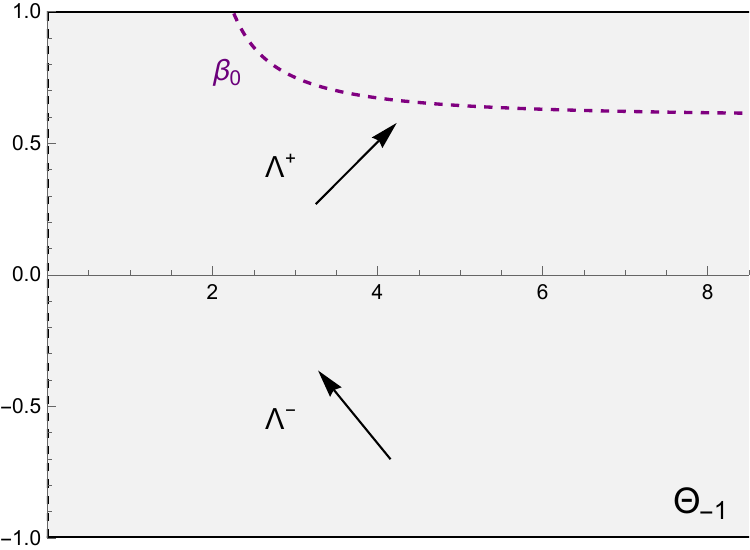}}}$
    \caption{The phase spaces $\Theta_1$ and $\Theta_{-1}$ for the choice $\h(t)=(t-0.6)^2$.}
    \label{Phasespace_pdesloc}
\end{figure}

By Proposition \ref{propsurfaceaxis}, let $\gamma_0(s) = (x_0(s),y_0(s))$ be the orbit that corresponds to the $\mathcal{H}$-surface that meets its rotation axis. Up to a change of orientation, suppose that $\gamma_0$ intersects the point $p_1 \in \Theta_1$. By monotonicity properties, for $x_0(s)>0$, we get that $\gamma_0$ lies in the $\Lambda_2^+$ region, intersect the curve $\Gamma_1$ and, then, stays in the region $\Lambda_3^+$ converging at infinity to the straight line $y=t_0$. For $x_0(s)<0$, we get that $\gamma_0$ converges to the equilibrium $-e_0$. See Figure \ref{paraboladesceixo}.

\begin{figure}
    \centering
    $\vcenter{\hbox{\includegraphics[scale=0.51]{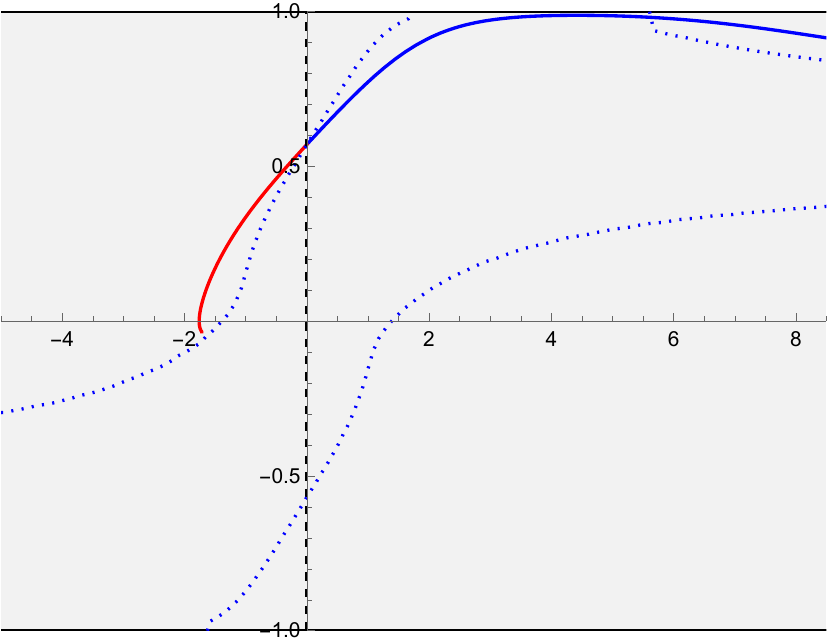}}}
    \hspace{0.8cm}
    \vcenter{\hbox{\includegraphics[scale=0.72]{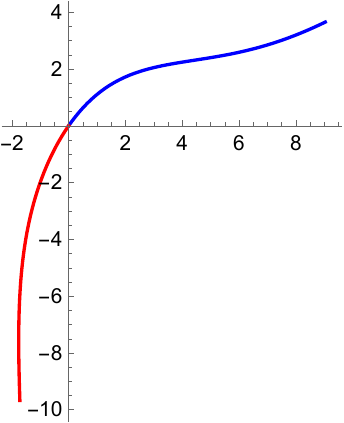}}}
    \hspace{0.8cm}
    \vcenter{\hbox{\includegraphics[scale=0.46]{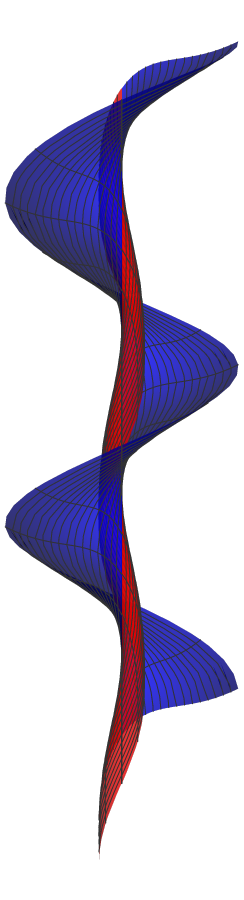}}}$
    \caption{The orbit $\gamma_0(s)$ in $\Theta_1$, where the blue component correspond to $x_0(s)>0$ and the red component to $x_0(s)<0$, its associated profile curve and its resulting helicoidal $\mathcal{H}$-surface, for the choice $\h(t) = (t-0.6)^2.$}
    \label{paraboladesceixo}
\end{figure}

Now, let $\gamma_1(s)$ be an orbit in $\Theta_1$ that passes through a point $(x_1,0) \in \Theta_1$, for some $0<x_1<\frac{1}{2\mathfrak{h}(0)}$, at some moment $s=s_0$. For $s>s_0$, we get by monotonicity properties that $\gamma_1$ lies in the $\Lambda_2^+$ region, intersect the curve $\Gamma_1$ and, then, stays in the region $\Lambda_3^+$ converging at infinity to the straight line $y=t_0$. For $s<s_0$, we get that $\gamma_1$ converges to the equilibrium $e_0$ as $s\to -\infty$. See Figure \ref{paraboladescpos}.

\begin{figure}
    \centering
    $\vcenter{\hbox{\includegraphics[scale=0.51]{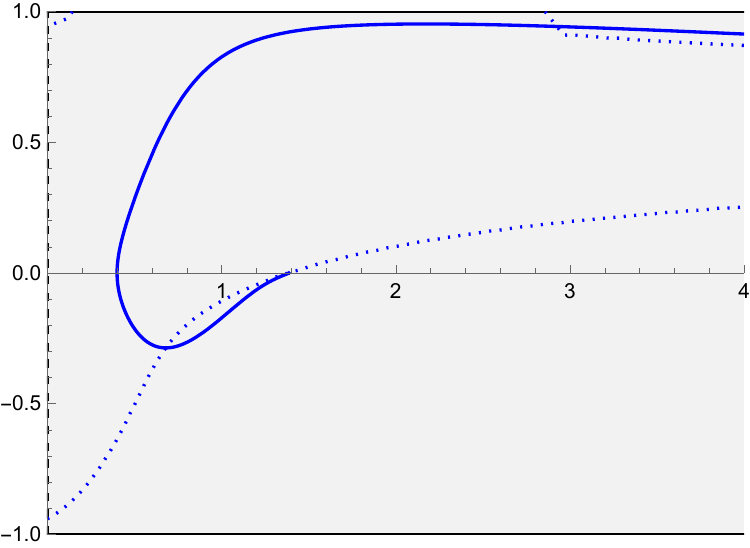}}}
    \hspace{0.8cm}
    \vcenter{\hbox{\includegraphics[scale=0.64]{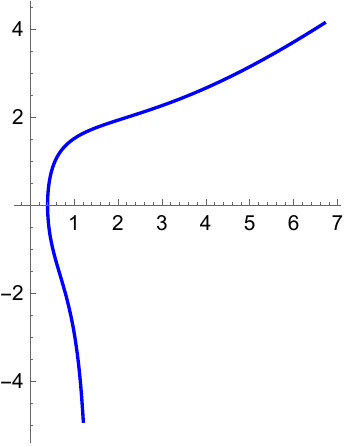}}}
    \hspace{0.8cm}
    \vcenter{\hbox{\includegraphics[scale=0.38]{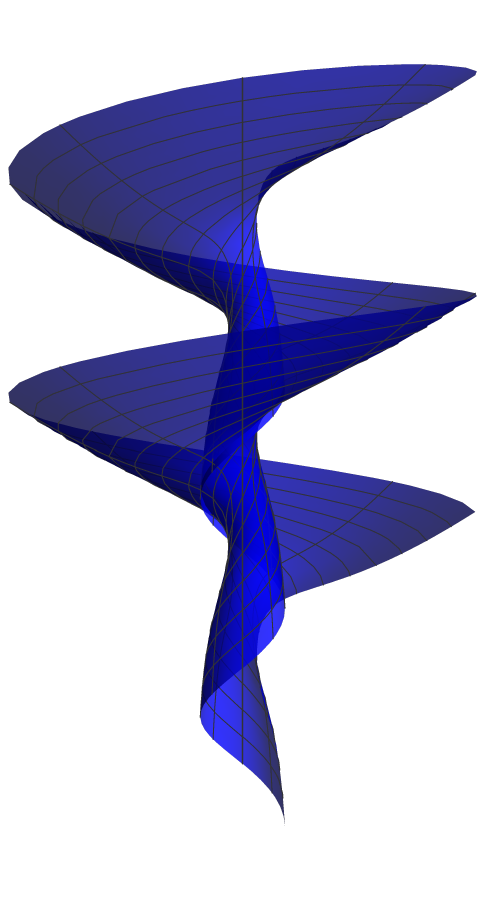}}}$
    \caption{The orbit $\gamma_1$ in $\Theta_1$, its associated profile curve and its resulting helicoidal $\mathcal{H}$-surface, for the choice $\h(t)=(t-0.6)^2$.}
    \label{paraboladescpos}
\end{figure}

Let now $\gamma_{-1}(s)$ be an orbit in $\Theta_{-1}$. By previous discussion, note that there exist some moments $s_1, s_2 \in \R$, with $s_1<s_2$, such that $\gamma_{-1}(s_1)=(x_1,-1)$ and $\gamma_{-1}(s_2)=(x_2,1)$, for some $x_1,x_2 >0$. We can analyse orbits in $\Theta_1$ that have these points as limit endpoints and smoothly glue the resulting $\mathcal{H}$-surfaces. Let $\gamma_{-1}^1$ be the orbit in $\Theta_1$ that has $(x_1,-1)$ as endpoint. By monotonicity properties, we obtain that $\gamma_{-1}^1$ converges to the equilibrium $e_0 \in \Theta_1$ as $s \to -\infty$. 

Next, let $\gamma_{-1}^2$ be the orbit in $\Theta_2$ that has $(x_2,1)$ as a limit endpoint. There are two possibilities for the behavior of $\gamma_{-1}^2$. If $x_2 \geq \frac{3|c_0|}{4}$, that is, $(x_2,1)$ is $p_0^+$ or is to the right of $p_0^+$, then $\gamma_{-1}^2$ stays in the region $\Lambda_3^+$ and converges at infinity to the straight line $y=t_0$. See Figure \ref{paraboladescneg}.

\begin{figure}
    \centering
    $\vcenter{\hbox{\includegraphics[scale=0.38]{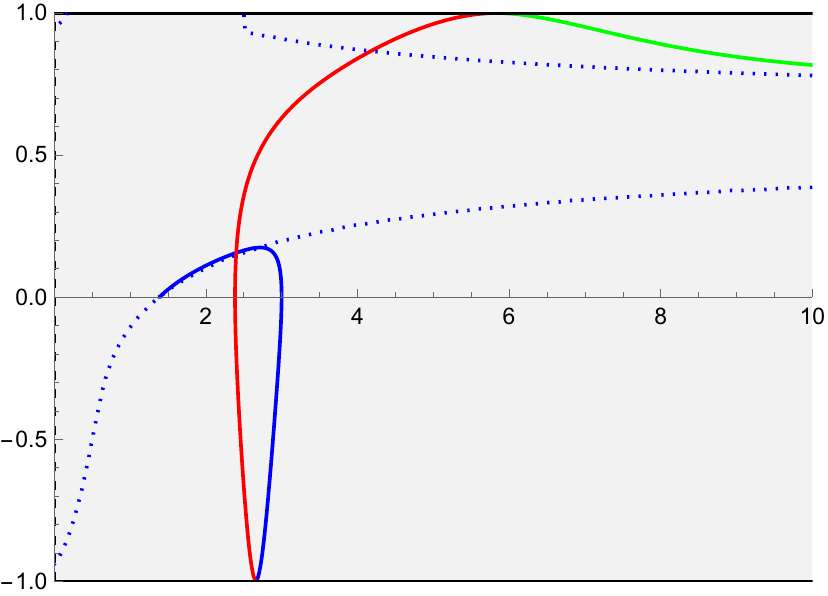}}}
    \hspace{0.6cm}
    \vcenter{\hbox{\includegraphics[scale=0.7]{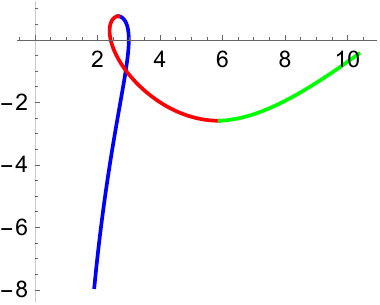}}}
    \hspace{0.6cm}
    \vcenter{\hbox{\includegraphics[scale=0.31]{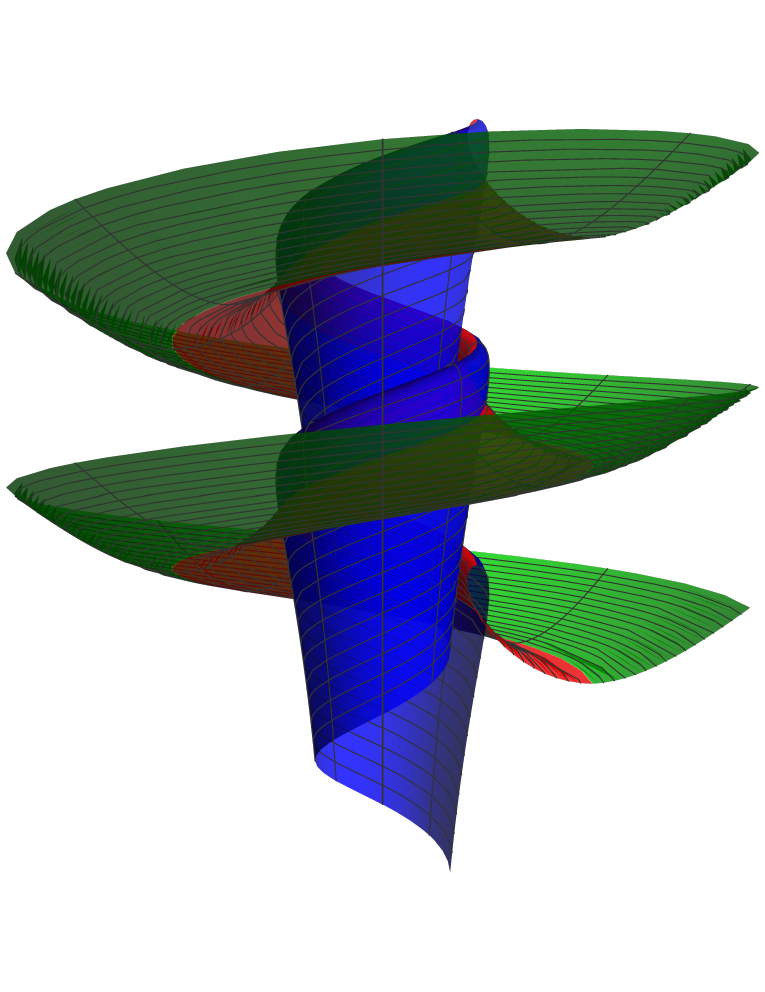}}}$
    \caption{On the left, the red curve corresponds to the orbit $\gamma_{-1}$ in $\Theta_{-1}$, the blue curve to the orbit $\gamma_{-1}^1$ in $\Theta_1$ and the green curve to the orbit $\gamma_{-1}^2$ in $\Theta_1$, with $x_2 \geq \frac{3|c_0|}{4}$, for the choice $\h(t) = (t-0.6)^2$. On the middle, their associated profile curve and, on the right, their resulting helicoidal $\mathcal{H}$-surface.}
    \label{paraboladescneg}
\end{figure}

Now, if $x_2 < \frac{3|c_0|}{4}$, that is, $(x_2,1)$ is to the left of $p_0^+$, then, by monotonicity properties, the orbit $\gamma_{-1}^2$ lies in the region $\Lambda_1^+$ around $(x_2,1)$, intersects the curve $\Gamma_1$ moving to the $\Lambda_2^+$ region, intersects again the curve $\Gamma_1$ and stays in the region $\Lambda_3^+$, converging at infinity to the straight line $y=t_0$. See Figure \ref{paraboladescneg2}.

\begin{figure}
    \centering
    $\vcenter{\hbox{\includegraphics[scale=0.62]{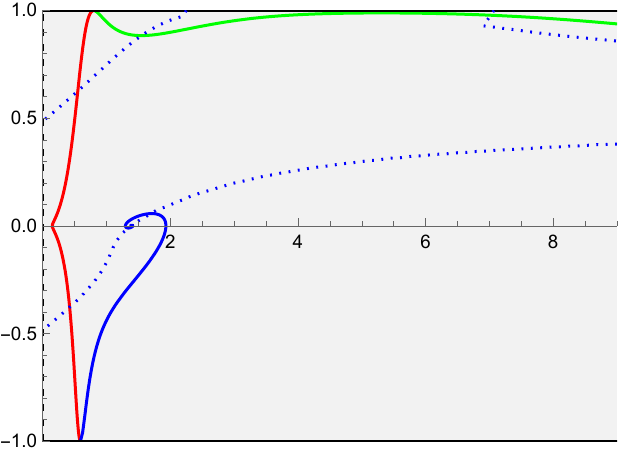}}}
    \hspace{0.8cm}
    \vcenter{\hbox{\includegraphics[scale=0.44]{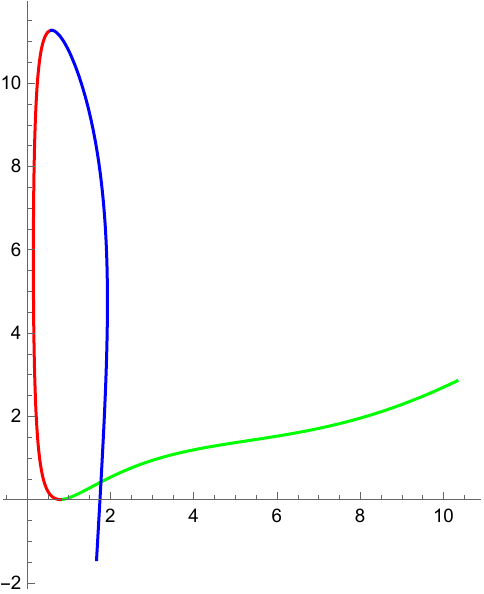}}}
    \hspace{0.8cm}
    \vcenter{\hbox{\includegraphics[scale=0.33]{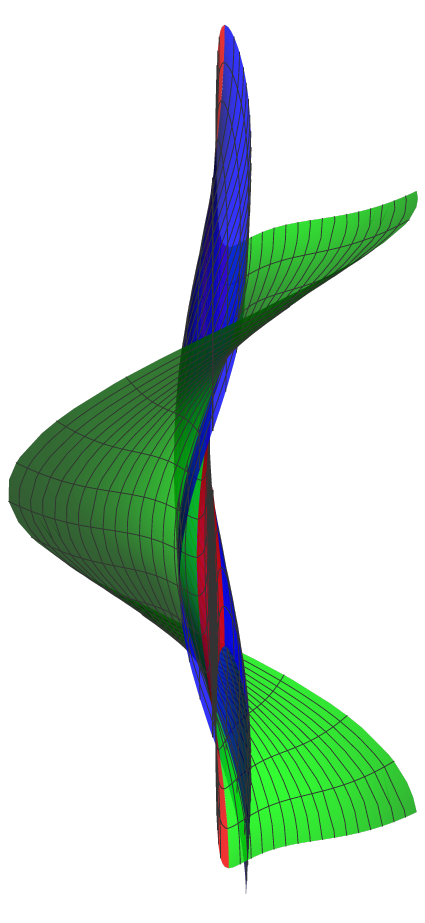}}}$
    \caption{On the left, the red curve corresponds to the orbit $\gamma_{-1}$ in $\Theta_{-1}$, the blue curve to the orbit $\gamma_{-1}^1$ in $\Theta_1$ and the green curve to the orbit $\gamma_{-1}^2$ in $\Theta_1$, with $x_2 < \frac{3|c_0|}{4}$, for the choice $\h(t) = (t-0.6)^2$. On the middle, their associated profile curve and, on the right, their resulting helicoidal $\mathcal{H}$-surface.}
    \label{paraboladescneg2}
\end{figure}

Finally, consider the example $\h(t)=(t-0.5)(t+2)$, that is, $t_0 = 0.5$. Thus,
$$p_0^+ = \left( \frac{|c_0|0.5}{\sqrt{1-0.5^2}},1\right) = \left(\frac{|c_0|\sqrt{3}}{3},1\right) \quad \text{and} \quad p_{-1} = \left(0, \frac{1}{\sqrt{1+c_0^2}} \right).$$

Note that, in this case, we have that $\h > 0$ above the curve $\beta_0$ and $\h < 0$ below it. This means that the curve $\Gamma_1$ is connected and stays above $\beta_0$, starting from the point $p_0^+$ and converging at infinity to the straight line $y=t_0$. Therefore, the phase space $\Theta_1$ has three monotonicity regions $\Lambda_1^+,\Lambda_2^+,\Lambda_3^+$, where $\Lambda_1^+ := \Theta_1 \cap \{ y < 0 \}$ and $\Lambda_2^+$ and $\Lambda_3^+$ are contained in the region $\{y>0\}$, with $\Lambda_3^+$ above $\Gamma_1$.

Moreover, the curve $\Gamma_{-1}$ has two connected components, where the first one connects the points $p_{-1}$ and $p_0^+$ and the second one contains the equilibrium $e_0$. Therefore, the phase space $\Theta_{-1}$ has five monotonicity regions $\Lambda_1^-,\dots,\Lambda_5^-$, three of them above the $y=0$ axis and two of them below it. See Figure \ref{Phasespace_iddesloc}.

\begin{figure}[h]
    \centering
    $\vcenter{\hbox{\includegraphics[scale=0.55]{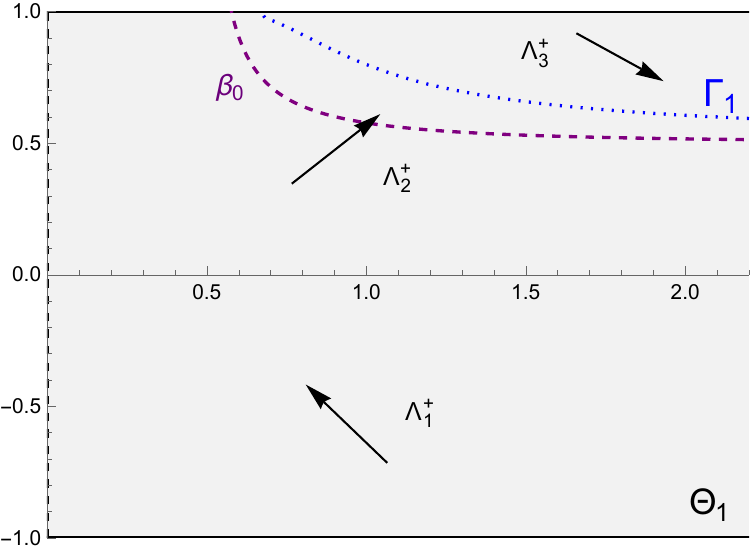}}}
    \hspace{0.7cm}
    \vcenter{\hbox{\includegraphics[scale=0.55]{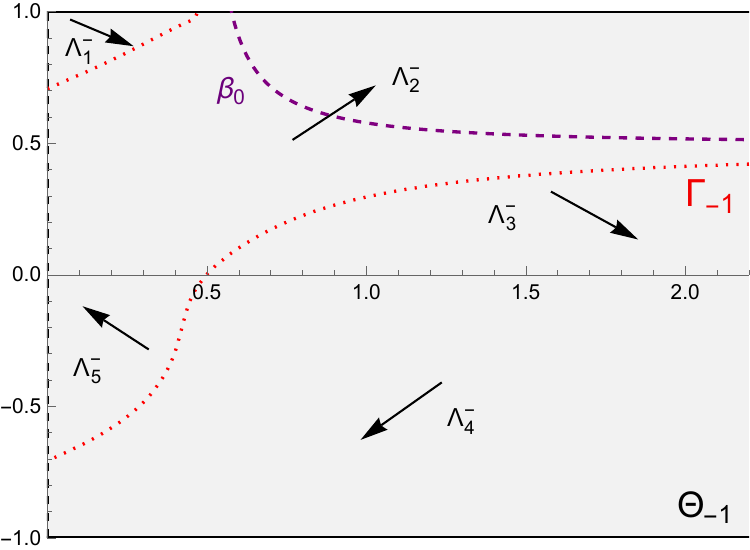}}}$
    \caption{The phase spaces $\Theta_1$ and $\Theta_{-1}$ for the choice $\h(t)=(t-0.5)(t+2)$.}
    \label{Phasespace_iddesloc}
\end{figure}

By Proposition \ref{propsurfaceaxis}, let $\gamma_0(s) = (x_0(s),y_0(s))$ be the orbit that corresponds to the $\mathcal{H}$-surface that meets its rotation axis. Up to a change of orientation, suppose that $\gamma_0$ intersects the point $p_{-1} \in \Theta_{-1}$. By monotonicity properties, for $x_0(s)<0$, we get that $\gamma_0$ converges to the equilibrium $-e_0$ as $s\to - \infty$. For $x_0(s)>0$, we get that $\gamma_0$ lies in the $\Lambda_2^-$ region and, at some moment $s_0 \in \R$, goes to a limit endpoint $(x_1,1)$, with $x_1 \geq \frac{|c_0|\sqrt{3}}{3}$. Considering $\gamma_0^+$ the orbit in $\Theta_1$ that also has $(x_1,1)$ as an endpoint, we can smoothly glue the resulting $\mathcal{H}$-surfaces. It follows from monotonicity properties that $\gamma_0^+$ stays in the $\Lambda_3^+$ region and converges at infinity to the straight line $y=t_0$. See Figure \ref{identidadedesloceixo}.

\begin{figure}
    \centering
    $\vcenter{\hbox{\includegraphics[scale=0.45]{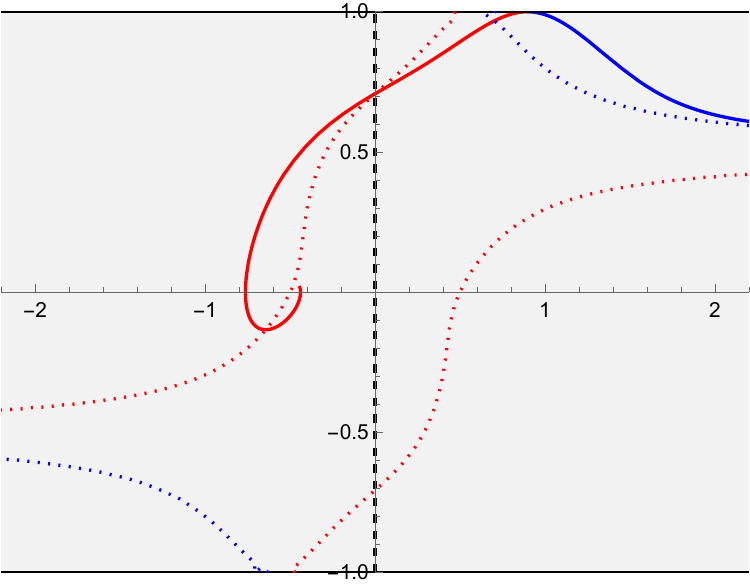}}}
    \hspace{0.8cm}
    \vcenter{\hbox{\includegraphics[scale=0.6]{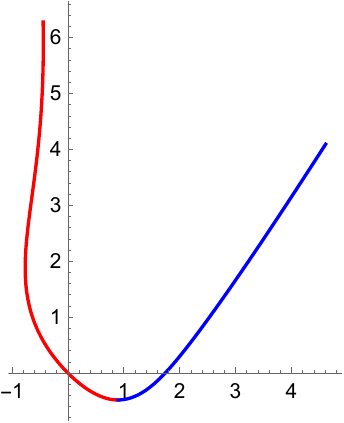}}}
    \hspace{0.8cm}
    \vcenter{\hbox{\includegraphics[scale=0.35]{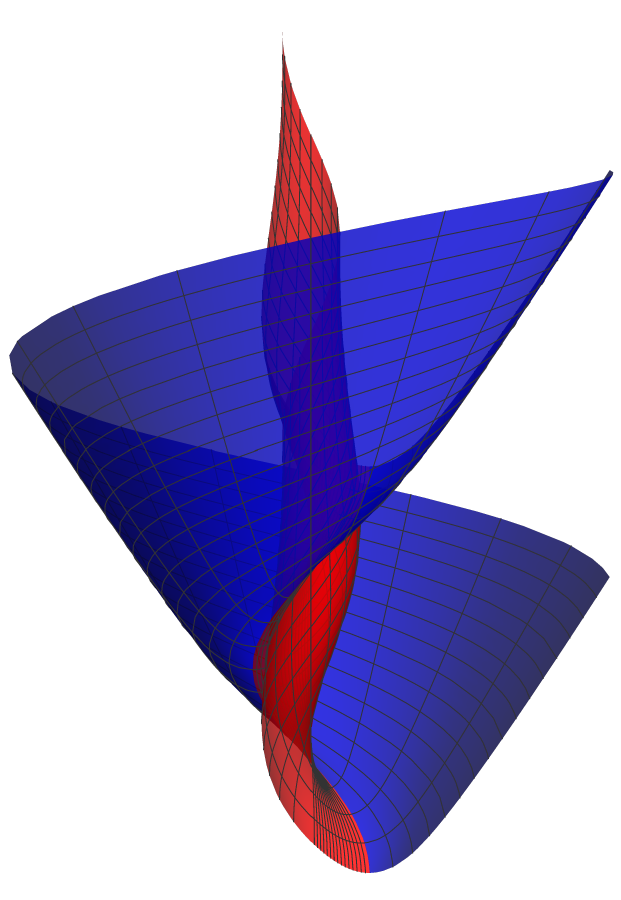}}}$
    \caption{On the left, the red curve corresponds to the orbit $\gamma_0$ in $\Theta_{-1}$ and the blue curve to the orbit $\gamma_0^+$ in $\Theta_1$, for the choice $\h(t) = (t-0.5)(t+2)$. On the middle, their associated profile curve and, on the right, their resulting helicoidal $\mathcal{H}$-surface.}
    \label{identidadedesloceixo}
\end{figure}

Let $\gamma_1(s)$ be an orbit in $\Theta_1$ that passes through a point $(x_0,0) \in \Theta_1$, for some $x_0>0$, at some moment $s=s_0$. By previous discussion and Remark \ref{convergeborda}, note that there exists $s_1 \in \R$, with $s_1 < s_0$, such that $\gamma_1(s_1) = (x_1,-1)$, for some $x_1 > x_0$. Consider $\gamma_1^-$ the orbit in $\Theta_{-1}$ that has $(x_1,-1)$ as an endpoint, which we can smoothly glue together the resulting $\mathcal{H}$-surfaces. By monotonicity properties, $\gamma_1^-$ converges to the equilibrium $e_0$ as $s \to -\infty$.

For $s > s_0$, there are two possibilities for the orbit $\gamma_1$ behavior. The first one is that $\gamma_1$ intersects the curve $\Gamma_1$ and, then, stays in the $\Lambda_3^+$ converging at infinity to the straight line $y=t_0$. See Figure \ref{identidadedeslocpos}.

\begin{figure}
    \centering
    $\vcenter{\hbox{\includegraphics[scale=0.5]{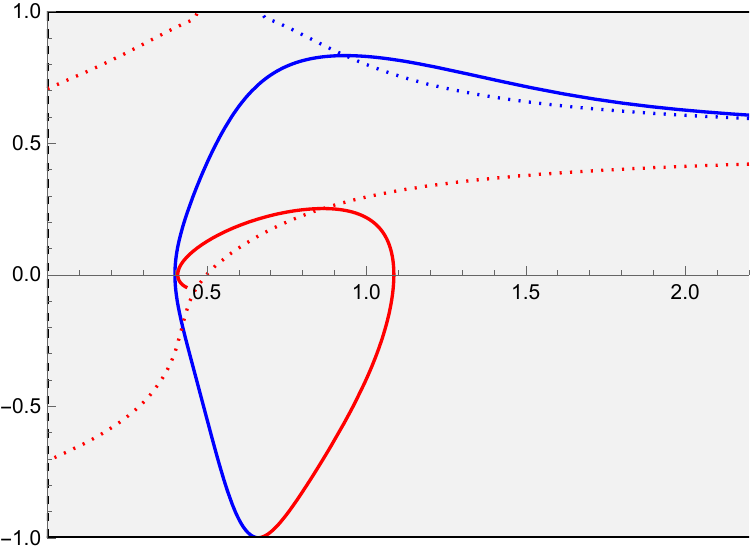}}}
    \hspace{0.8cm}
    \vcenter{\hbox{\includegraphics[scale=0.55]{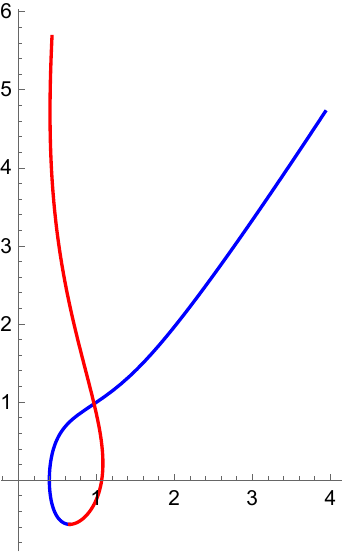}}}
    \hspace{0.8cm}
    \vcenter{\hbox{\includegraphics[scale=0.35]{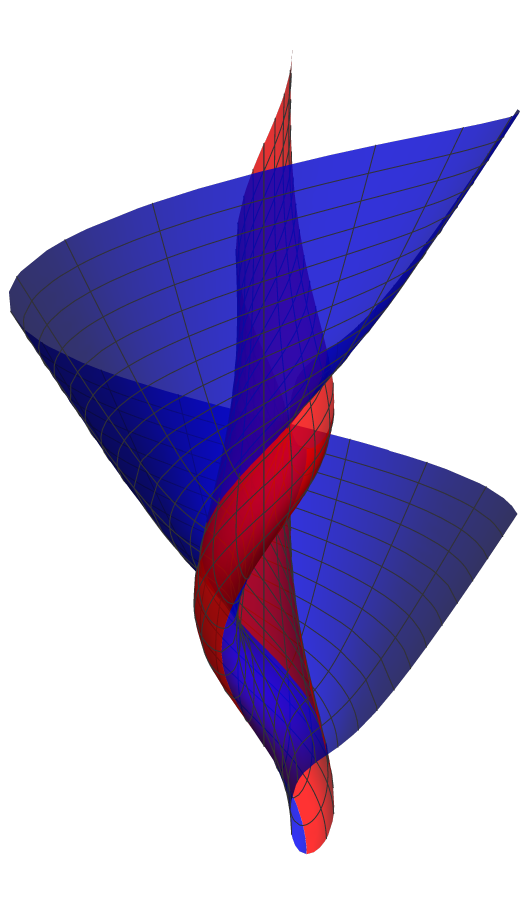}}}$
    \caption{On the left, the blue curve corresponds to the orbit $\gamma_1$ in $\Theta_1$ and the red curve to the orbit $\gamma_1^-$ in $\Theta_{-1}$, for the choice $\h(t) = (t-0.5)(t+2)$. On the middle, their associated profile curve and, on the right, their resulting helicoidal $\mathcal{H}$-surface.}
    \label{identidadedeslocpos}
\end{figure}

The second possibility is that the orbit $\gamma_1$ has an endpoint $(x_2,1)$, with $x_2 < \frac{|c_0|\sqrt{3}}{3}$, that is, $(x_2,1)$ is to the left of $p_0^+$. Consider $\gamma_1^2$ the orbit in $\Theta_{-1}$ that has $(x_2,1)$ as an endpoint, which we can smoothly glue together the resulting $\mathcal{H}$-surfaces. Thus, we get, by monotonicity properties, that $\gamma_1^2$ lies in the $\Lambda_1^-$ region for points near $(x_2,1)$, intersects the curve $\Gamma_{-1}$ at some moment moving to the $\Lambda_2^-$ region and, then, goes to another endpoint $(x_3,1)$, with $x_3 > \frac{|c_0|\sqrt{3}}{3}$. Let $\gamma_1^3$ be the orbit in $\Theta_1$ that has $(x_3,1)$ as an endpoint. Observe that $\gamma_1^3$ stays in the $\Lambda_3^+$ region and converges at infinity to the straight line $y=t_0$. See Figure \ref{identidadedeslocneg}.

\begin{figure}
    \centering
    $\vcenter{\hbox{\includegraphics[scale=0.52]{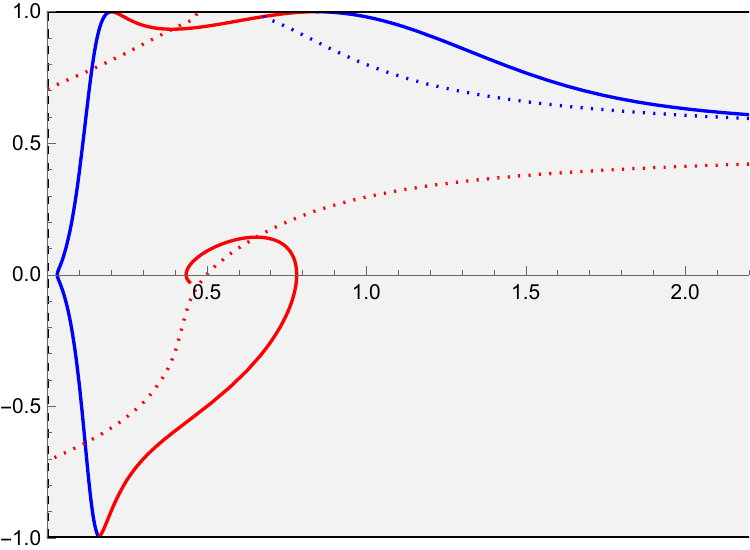}}}
    \hspace{0.8cm}
    \vcenter{\hbox{\includegraphics[scale=0.52]{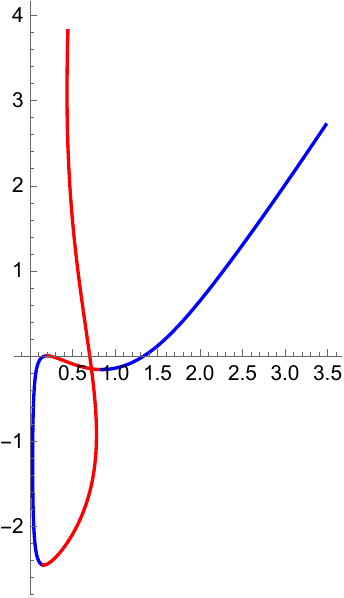}}}
    \hspace{0.8cm}
    \vcenter{\hbox{\includegraphics[scale=0.35]{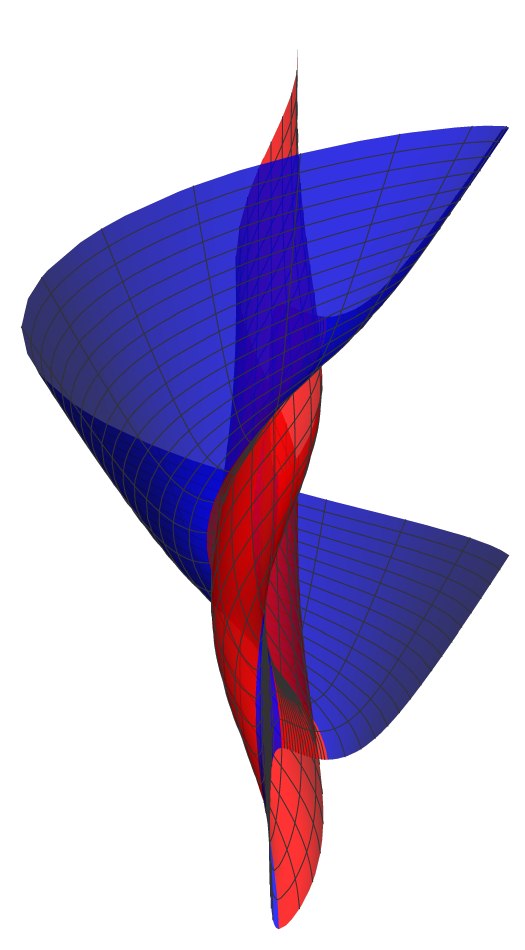}}}$
    \caption{On the left, the blue curves correspond to the orbits $\gamma_1$ and $\gamma_1^3$ in $\Theta_1$ and the red curves to the orbits $\gamma_1^-$ and $\gamma_1^2$ in $\Theta_{-1}$, for the choice $\h(t) = (t-0.5)(t+2)$. On the middle, their associated profile curve and, on the right, their resulting helicoidal $\mathcal{H}$-surface.}
    \label{identidadedeslocneg}
\end{figure}

Lastly, let $\gamma_{-1}(s)$ be an orbit in $\Theta_{-1}$ that passes at some moment $s=s_1$ through a point $(x_{-1},0)$, with $0<x_{-1}<\frac{1}{2\mathfrak{h}(0)}$, that is, $(x_{-1},0)$ is to the left of the equilibrium $e_0 \in \Theta_{-1}$. By monotonicity properties, the orbit $\gamma_{-1}(s)$ converges to the equilibrium $e_0$ as $s \to -\infty$. Moreover, for $s>s_1$, the orbit $\gamma_{-1}$ stays in the $\Lambda_2^-$ region and has an endpoint $(x_4,1)$, with $x_4 > \frac{|c_0|\sqrt{3}}{3}$. Considering $\gamma_{-1}^+$ the orbit in $\Theta_1$ that has $(x_4,1)$ as an endpoint, which we can smoothly glue together the resulting $\mathcal{H}$-surfaces, we get that $\gamma_{-1}^+$ stays in the $\Lambda_3^+$ region and converges at infinity to the straight line $y=t_0$. See Figure \ref{identidadedeslocneg2}.

\begin{figure}[h]
    \centering
    $\vcenter{\hbox{\includegraphics[scale=0.48]{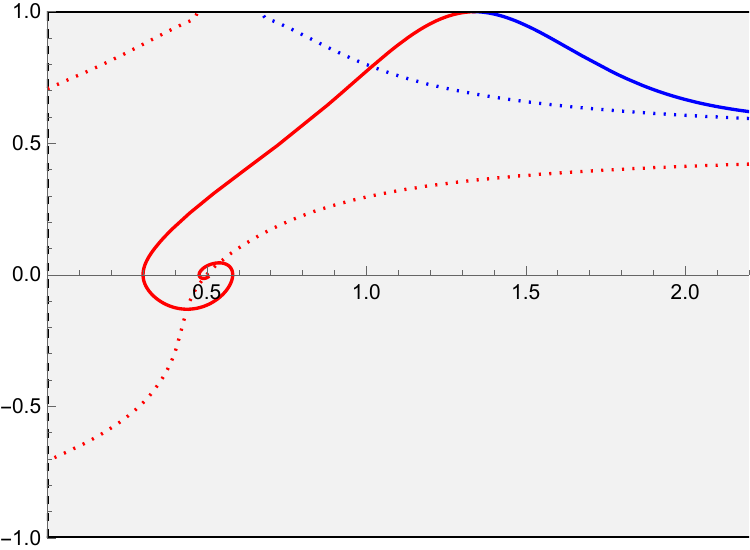}}}
    \hspace{0.8cm}
    \vcenter{\hbox{\includegraphics[scale=0.6]{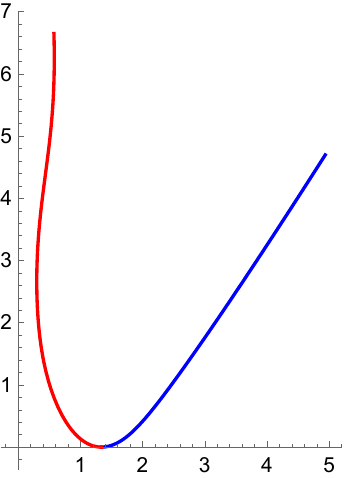}}}
    \hspace{0.8cm}
    \vcenter{\hbox{\includegraphics[scale=0.33]{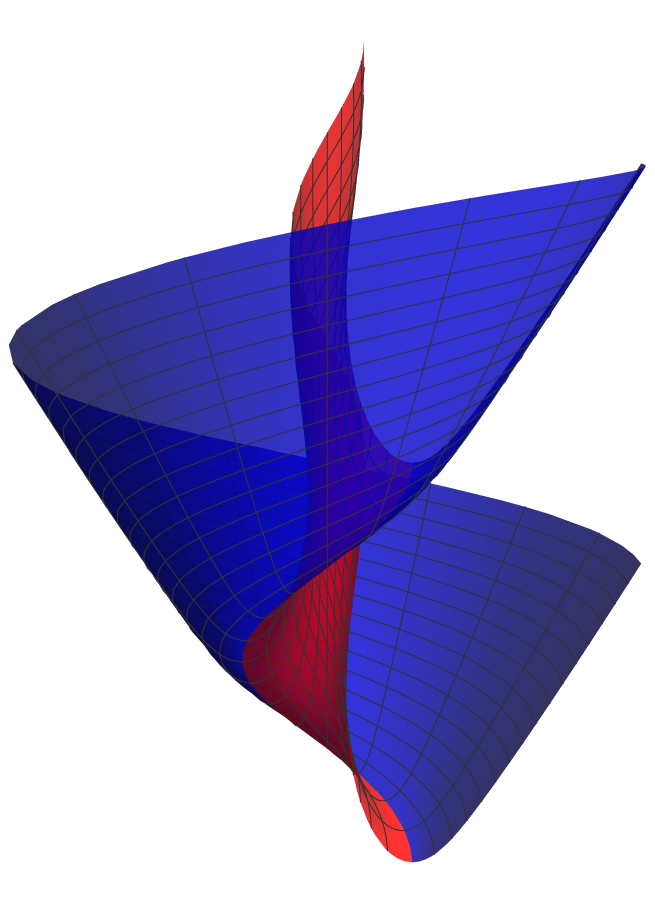}}}$
    \caption{On the left, the red curve corresponds to the orbit $\gamma_{-1}$ in $\Theta_{-1}$ and the blue curve to the orbit $\gamma_{-1}^+$ in $\Theta_1$, for the choice $\h(t) = (t-0.5)(t+2)$. On the middle, their associated profile curve and, on the right, their resulting helicoidal $\mathcal{H}$-surface.}
    \label{identidadedeslocneg2}
\end{figure}

\section*{Acknowledgements}

This work was done in Instituto de Matemáticas of University of Granada and the author is supported by BEPE - FAPESP (São Paulo Research Foundation) Grant number 2022/14381-5. The author is grateful to Prof. José A. Gálvez for his supervisory on the project and all his collaboration.






\vskip 0.2cm

\noindent Aires E. M. Barbieri

\noindent Instituto de Ciências Matemáticas e de Computação, \\ Universidade de São Paulo, São Carlos (Brazil).

\noindent  e-mail: {\tt airesb@usp.br}

\end{document}